\documentclass{daj}

\dajAUTHORdetails{%
  title = {A Bilinear Bogolyubov Argument in Abelian Groups}, 
  author = {Luka Mili\'cevi\'c},
  plaintextauthor = {Luka Milicevic},
  plaintexttitle = {A Bilinear Bogolyubov Argument in Abelian Groups}, 
  keywords = {Bogolyubov argument, bilinear maps, Bohr sets, lattice theory},
}   

\dajEDITORdetails{%
   year={2024},
   number={20},
   received={22 December 2021},   
   published={30 December 2024},  
   doi={10.19086/da.127776},       
}   

\usepackage{amsmath,amsthm}
\usepackage{amssymb}
\usepackage{enumerate}
\usepackage{bm}
\usepackage{bbm}
\usepackage{bold-extra}
\usepackage{mathtools}
\usepackage{relsize}
\usepackage{scalerel}
\usepackage{appendix}

\newcommand{\tss}[1]{\textsuperscript{#1}}

\newcommand{\dhor}{\mathsf{D}_{\mathsf{hor}}}
\newcommand{\dver}{\mathsf{D}_{\mathsf{ver}}}

\newcommand{\on}[1]{
	\operatorname{#1}
}

\newcommand{\tn}[1]{ 
	|#1|_{\mathbb{T}}
}

\newcommand{\btn}[1]{ 
	\Big|#1\Big|_{\mathbb{T}}
}

\newcommand\ssk[1]{
	\substack{#1}
}

\newcommand\ex{\mathop{\mathbb{E}}}

\newcommand{\exx}{
  \mathop{
    \mathchoice{\vcenter{\hbox{\larger[4]$\mathbb{E}$}}}
               {\kern0pt\mathbb{E}}
               {\kern0pt\mathbb{E}}
               {\kern0pt\mathbb{E}}
  }\displaylimits
}

\makeatletter
\newcommand*\bcdot{\mathpalette\bigcdot@{0.5}}
\newcommand*\bigcdot@[2]{\mathbin{\vcenter{\hbox{\scalebox{#2}{$\m@th#1\bullet$}}}}}
\makeatother

\allowdisplaybreaks

\newcommand\fco{\lbrack}
\newcommand\fcc{\rbrack^\wedge}

\newcommand\id{\mathbbm{1}}	

\newtheorem{theorem}{Theorem}
\newtheorem{corollary}[theorem]{Corollary}
\newtheorem{lemma}[theorem]{Lemma}
\newtheorem{proposition}[theorem]{Proposition}

\newtheorem{claim}[theorem]{Claim}

\theoremstyle{definition}

\newtheorem{rem}[theorem]{Remark}

\begin{document}

\begin{frontmatter}[classification=text]

\title{A Bilinear Bogolyubov Argument in Abelian Groups} 

\author[luka]{Luka Mili\'cevi\'c\thanks{This work was supported by the Serbian Ministry of Education, Science and Technological Development through Mathematical Institute of the Serbian Academy of Sciences and Arts.}}

\begin{abstract}
The bilinear Bogolyubov argument for $\mathbb{F}_p^n$ states that if we start with a dense set $A \subseteq \mathbb{F}_p^n \times \mathbb{F}_p^n$ and carry out sufficiently many steps where we replace every row or every column of $A$ by the set difference of it with itself, then inside the resulting set we obtain a bilinear variety of codimension bounded in terms of density of $A$. In this paper, we generalize the bilinear Bogolyubov argument to arbitrary finite abelian groups. Namely, if $G$ and $H$ are finite abelian groups and $A \subseteq G \times H$ is a subset of density $\delta$, then the procedure above applied to $A$ results in a set that contains a bilinear analogue of a Bohr set, with the appropriately defined codimension bounded above by $\log^{O(1)} (O(\delta^{-1}))$.
\end{abstract}
\end{frontmatter}

\section{Introduction}

Let us begin by recalling the following classical fact due to Bogolyubov. Whenever $A \subset \mathbb{Z}/N\mathbb{Z}$ is a subset of density $\alpha$, then the set $2A - 2A$ contains a Bohr set of bounded codimension and large radius. (We are being deliberately informal at the moment, we shall define Bohr sets properly slightly later in the introduction.) This fundamental fact was generalized to the bilinear setting in the case of vector spaces over finite fields by Bienvenu and L\^{e}~\cite{BienLe} and, independently, by Gowers and the author~\cite{bogPaper}. Hosseini and Lovett~\cite{HosseiniLovett} improved the bounds to the nearly optimal ones. Before stating their result, which we refer to as the bilinear Bogolyubov argument, we need to set up some notation. Let $A \subseteq \mathbb{F}_p^n \times \mathbb{F}_p^n$ be a subset. We write $\dhor A$ for the \emph{horizontal difference set} of $A$ which is defined as $\dhor A = \{(x_1 - x_2, y) \colon (x_1, y), (x_2, y) \in A\}$. Similarly, we write $\dver A$ for the \emph{vertical difference set} of $A$ which is defined as $\dver A = \{(x, y_1 - y_2) \colon (x, y_1), (x, y_2) \in A\}$. In other words, the horizontal difference set is obtained by taking the difference set inside each row and the vertical difference set is obtained by taking the difference set inside each column. Iterated directional difference sets are obtained by the obvious compositions, as $\dhor$ and $\dver$ are simply maps from the power-set $\mathcal{P}(\mathbb{F}_p^n \times \mathbb{F}_p^n)$ to itself.

\begin{theorem}[Bilinear Bogolyubov argument in $\mathbb{F}_p^n$]\label{bilinearBogRuzsaVS}Let $A \subseteq \mathbb{F}_p^n \times \mathbb{F}_p^n$ be a set of density $\delta > 0$. Then there exist subspaces $U, V \leq \mathbb{F}_p^n$ of codimension at most $O(\log^{64 + o(1)} \delta^{-1})$, a positive integer $r \leq O(\log^{64 + o(1)} \delta^{-1})$ and a bilinear map $\beta \colon U \times V \to \mathbb{F}_p^r$ such that 
\[\{(x,y) \in U \times V \colon \beta(x,y) = 0\} \subseteq \dhor\dver\dver\dhor \dver\dver \dver \dhor \dhor A.\]
\end{theorem}

\begin{rem}
    Using a recent breakthrough~\cite{Marton1,Marton2} of Gowers, Green, Manners and Tao, in which they proved a conjecture of Marton, also known as the polynomial Freiman--Ruzsa conjecture, the exponents of the log can be improved to some value significantly smaller than $64 + o(1)$.
\end{rem}

Sets of the form in the conclusion of the theorem above are called \emph{bilinear varieties}. Let us note that the bounds in~\cite{BienLe} and~\cite{bogPaper} were towers of two or three exponentials in $\log^{O(1)} \delta^{-1}$, but had fewer applications of the directional difference sets.\\

Theorem~\ref{bilinearBogRuzsaVS} was used by Bienvenu and L\^{e}~\cite{BienLe2} to show that the M\"{o}bius function does not have significant correlation with quadratic phases. In~\cite{U4paper}, Gowers and the author used a more analytic version of the theorem above as an important ingredient in a proof of a quantitative inverse theorem for the $\|\cdot\|_{\mathsf{U}^4}$ uniformity norm in finite dimensional vector spaces over $\mathbb{F}_p$ for $ p\geq 5$. Finally, Theorem~\ref{bilinearBogRuzsaVS} was generalized to the multilinear setting by the same authors in~\cite{genPaper}.\\

The goal of this paper is to prove a version of bilinear Bogolyubov argument in arbitrary finite abelian groups with nearly optimal bounds. In order to state it, we need a couple of (mostly standard) definitions. For a finite abelian group $G$, its \emph{dual group} is defined as the additive group of \emph{characters} on $G$, where a character is a homomorphism from $G$ to the unit circle $\mathbb{T} = \mathbb{R}/\mathbb{Z}$. We also make use of a function $\tn{\cdot} \colon \mathbb{T} \to [0,1/2]$ that maps elements of the unit circle to their distance to $0 + \mathbb{Z}$. For a set of characters $\Gamma \subseteq \hat{G}$ and a non-negative real $\rho$, we define the Bohr set $B(\Gamma; \rho)$ to be the set of all elements $x \in G$ such that $\tn{\gamma(x)} \leq \rho$ for all $\gamma \in \Gamma$. The set $\Gamma$ is called the \emph{frequency set} and $\rho$ is the \emph{radius} of the Bohr set. We also say that the \emph{codimension} of $B(\Gamma; \rho)$ is at most $|\Gamma|$. Finally, given a map $\phi \colon A \to H$ from a subset $A$ of an abelian group $G$ to another abelian group $H$, we say that it is \emph{Freiman-linear} if $\phi(a - b) = \phi(a) - \phi(b)$ whenever $a, b, a-b \in A$. Freiman-linearity is simply a linearized version of the property of being a Freiman homomorphism. We are now ready to state the main result of this paper. 

\begin{theorem}[Bilinear Bogolyubov argument in finite abelian groups]\label{bogruzsabilinearintro}Let $G$ and $H$ be finite abelian groups and let $A \subseteq G \times H$ be a set of density $\delta$. Then there exist a positive quantity $\rho \geq \exp\Big(-\log^{O(1)} (10 \delta^{-1})\Big)$, sets $\Gamma \subseteq \hat{G},  \Psi \subseteq \hat{H}$ of size at most $\log^{O(1)} (10\delta^{-1})$ and Freiman-linear maps $L_1, \dots,$ $L_r \colon B(\Psi; \rho) \to \hat{G}$ for some positive integer $r \leq \log^{O(1)} (10 \delta^{-1})$ such that the bilinear Bohr variety
\[\Big\{(x,y) \in B(\Gamma; \rho) \times B(\Psi; \rho) \colon x \in B(L_1(y), \dots, L_r(y); \rho) \Big\}\]
is contained inside $\dhor\dver \dver \dhor \dver \dhor \dhor A$.\end{theorem}

\vspace{\baselineskip}
\noindent\textbf{Remark on logarithms.} In this paper, logarithms appear exclusively in upper bounds (even in $\exp\big(-\log^{O(1)} (10 \delta^{-1})\big)$ there is an upper bound in the argument of $\exp$). To ease the notational burden in the bounds that appear quite frequently, from now on $\log x$ for $x \geq 1$ stands for $\log x + 2$. The main reason for doing so is that it guarantees that $\log$ is sufficiently larger than 1.\\

\noindent\textbf{Remark on bounds.} In the vector space case, Theorem~\ref{bilinearBogRuzsaVS} gives a bound of the form $\log_p (p^n / |U|) + \log_p (p^n /|V|) + r \leq O_p(\log^{O(1)} \delta^{-1})$. Since the size of the Bohr set $B(\Gamma; \rho)$ is comparable to $\rho^{|\Gamma|}|G|$, the corresponding bound in the general case would be $(|\Gamma| + |\Psi|) \log \rho^{-1} + r  \leq O(\log^{O(1)} \delta^{-1})$, which is indeed the case in Theorem~\ref{bogruzsabilinearintro}.\\ 
\indent In contrast to Theorem~\ref{bilinearBogRuzsaVS}, we opted not to give explicit exponents of $\log (10\delta^{-1})$. In comparison with the case of finite vector spaces, there are additional sources of poly-logarithmic factors in general abelian groups. For example, Theorem~\ref{dualsidentityintro}, which generalizes the vector space identity $(U \cap V)^\perp = U^\perp + V^\perp$, Theorem~\ref{latticesIntro}, which is a quantitative version of the fundamental theorem of lattices, and Proposition~\ref{cosettobohrset}, which locates Bohr sets inside coset-progressions, all introduce such losses in bounds, while also reducing to trivial linear-algebraic facts in the finite vector spaces case. In particular, the exponent in Theorem~\ref{bogruzsabilinearintro} would be significantly higher than 64.\\

Let us compare the conclusions of Theorems~\ref{bilinearBogRuzsaVS} and~\ref{bogruzsabilinearintro}. As usual, we think of Bohr sets as replacement for subgroups in general finite abelian groups, so having $B(\Gamma; \rho) \times B(\Psi; \rho)$ in Theorem~\ref{bogruzsabilinearintro} instead of $U \times V$ in Theorem~\ref{bilinearBogRuzsaVS} is expected. When it comes to the bilinear map $\beta \colon U \times V \to \mathbb{F}_p^r$ in Theorem~\ref{bilinearBogRuzsaVS}, if we fix a dot product $\cdot$ on $\mathbb{F}_p^n$, then linear forms on $\mathbb{F}_p^n$ are exactly the maps $v \mapsto a \cdot v$ for some fixed $a \in \mathbb{F}_p^n$. In particular, for each $i \in [r]$, we can find a linear map $L_i \colon \mathbb{F}_p^n \to \mathbb{F}_p^n$ such that $\beta_i(x,y) = x \cdot L_i(y)$. Let us interpret expressions $\frac{\lambda}{p}$ for $\lambda \in \mathbb{F}_p$ as elements of $\mathbb{T}$, namely as $\frac{k}{p} + \mathbb{Z}$ where $k \in \{0,1,\dots, p-1\} \subseteq \mathbb{Z}$ is the integer in the given interval corresponding to $\lambda$. With these conventions we see that
\[\{(x,y) \in U \times V \colon \beta(x,y) = 0\} = \Big\{(x,y) \in U \times V \colon \btn{\frac{x \cdot L_i(y)}{p}} \leq 0\Big\}.\]

Finally, noting that $v \mapsto \frac{a \cdot v}{p}$ is a character on $\mathbb{F}_p^n$ we see that Theorem~\ref{bogruzsabilinearintro} implies Theorem~\ref{bilinearBogRuzsaVS} (though with somewhat weaker bounds as the implicit constant in $\log^{O(1)} \delta^{-1}$ is probably higher than 80, but we are mainly interested in the same form of bounds).\\
\indent As a final comment on the relationship between the two results, Freiman-linear maps on Bohr sets seem to be the right generalization of linear maps in this context.\footnote{See Theorem~\ref{approxFreimanHom}. We use coset progressions as the domain, but coset progressions and Bohr sets are very closely related, see also Proposition~\ref{cosettobohrset}.}\\

\noindent\textbf{A comparison with earlier works and an outline.} Let us briefly compare the previous works on the bilinear Bogolyubov argument~\cite{BienLe},~\cite{bogPaper} and~\cite{HosseiniLovett}. Let us fix a dense set $A \subseteq \mathbb{F}_p^n \times \mathbb{F}_p^n$. Very roughly speaking, all three proofs have the following structure.
\begin{itemize}
\item[\textbf{Step 1.}] Apply the Bogolyubov argument in all dense rows, to find a dense subset $Y \subseteq \mathbb{F}_p^n$ such that for each $y \in Y$, after taking double difference set in the row indexed by $y$, we obtain a bounded codimension vector space $V_y$.
\item[\textbf{Step 2.}] Take difference sets in columns and rows sufficient number of times to find a further dense subset $Y' \subseteq \mathbb{F}_p^n$ and a bounded number of linear maps $L_1, \dots, L_r \colon \mathbb{F}_p^n \to \mathbb{F}_p^n$ such that for each $y \in Y'$ we now obtain the subspace $\langle L_1(y), \dots, L_r(y)\rangle^\perp$ in the row indexed by $y$.
\item[\textbf{Step 3.}] Apply the Bogolyubov argument to $Y'$ to obtain a subspace $U \subseteq 2Y' - 2Y'$ and use this subspace to induce the desired structure in columns of the iterated difference set of $A$.
\end{itemize}

The way this strategy is carried out is different in each of the three papers. However, due to its simplicity, \textbf{Step 1} is essentially the same in all three works. Furthermore, in \textbf{Step 2} all three papers use the fact that the maps that respect many additive quadruples coincide on a large set with affine maps (see Theorem~\ref{approxFreimanHom} for a precise statement).\\

As for differences, which are most obvious in \textbf{Step 3}, the work of Bienvenu and L\^{e}~\cite{BienLe} uses the arithmetic regularity lemma, Gowers and the author use a direct Fourier analytic argument to conclude the proof, while Hosseini and Lovett perform a careful algebraic manipulation of the resulting subspaces in the rows of the iterated difference set and use a significant amount of linear algebra. There is also an important difference in how Hosseini and Lovett find the linear maps in \textbf{Step 2} without a loss in efficiency; we shall discuss their idea when we reach the corresponding step in the proof in this paper. Overall, the papers~\cite{bogPaper} and~\cite{HosseiniLovett} are more similar in their organization and this paper is built upon ideas in those two, so we focus on them in the rest of the outline.\\

Our argument will follow that of Hosseini and Lovett in the first two steps, but we shall need to diverge in order to complete \textbf{Step 3}. Their argument involves a consideration of the rank of linear combinations of linear maps $L_1, \dots, L_r$ produced in \textbf{Step 2}, which is not readily available in our more general setting. However, the high-rank condition on the system of linear maps is essentially equivalent to the bilinear variety
\[\Big\{(x,y) \in \mathbb{F}_p^n \times \mathbb{F}_p^n \colon (\forall i \in [r]) x \cdot L_i(y) = 0\Big\}\]
being quasirandom (see Corollary 5.2 and Lemma 5.4 in~\cite{U4paper}). With this in mind, we prove an algebraic regularity lemma for bilinear Bohr varieties.\\
\indent Recall that a \emph{coset progression} in an abelian group $G$, introduced by Green and Ruzsa in their generalization of Freiman's theorem to arbitrary abelian groups~\cite{greenRuzsaFreiman}, is a set $C$ of the form $L_1 + \dots + L_r + H$, where $L_i$ are arithmetic progressions and $H \leq G$ is a subgroup. We refer to the number $r$ as the \emph{rank} of the coset progression. We say that a coset progression $C$ is \emph{proper} if $|C| = |L_1| \cdots |L_r||H|$, (thus all sums of $(r + 1)$-tuples are distinct), and we say that $C$ is \emph{symmetric} if all $L_i$ are symmetric, i.e. $L_i = -L_i$. We use coset progressions instead of Bohr sets as domains of Freiman-linear maps in the statement as the coset progressions have nicer partitions than Bohr sets, but this is of minor importance.

\begin{theorem}[Algebraic regularity lemma for bilinear Bohr varieties]\label{algreglemmaintro}Let $G$ and $H$ be finite abelian groups. Let $C$ be a symmetric proper coset progression of rank $d$ inside the group $H$, let $\Gamma \subseteq \hat{G}$ and let $L_1, \dots, L_r \colon C \to \hat{G}$ be Freiman-linear maps. Let $\rho > 0$ and $\eta > 0$ be given. We may partition $C$ into further proper coset progressions $C_1, \dots, C_m$ of rank at most $d$, where
\[m \leq \exp\Big(d^{O(1)} r^{O(1)} |\Gamma|^{O(1)}\log^{O(1)} (\eta^{-1}) \log^{O(1)} (\rho^{-1})\Big),\]
such that for each $i \in [m]$ there exist positive reals $\delta_i$ and $\rho_i \in [\rho/2, \rho]$ such that the following two quasirandomness properties hold.
\begin{itemize}
\item[\textbf{(i)}] For at least a $1 - \eta$ proportion of all elements $y \in C_i$ we have 
\[\Big||B(\Gamma \cup \{L_1(y), \dots, L_r(y)\}; \rho_i)| - \delta_i |B(\Gamma; \rho_i)|\Big| \leq \eta |G|.\]
\item[\textbf{(ii)}] For at least a $1 - \eta$ proportion of all pairs $(y, y') \in C_i \times C_i$ we have
\[\Big||B(\Gamma \cup \{L_1(y), \dots, L_r(y), L_1(y'), \dots, L_r(y')\}; \rho_i)| - \delta^2_i |B(\Gamma; \rho_i)|\Big| \leq \eta |G|.\]
\end{itemize}
\end{theorem}

Using this result, we may finish \textbf{Step 3} using the well-known quasirandomness properties of bipartite graphs.\\

Let us note that Theorem~\ref{algreglemmaintro} is another important ingredient of the proof~\cite{U4paper} of the inverse theorem for the $\|\cdot\|_{\mathsf{U}^4}$ uniformity norm in $\mathbb{F}_p^n$ that is generalized in this paper. With this in mind, Theorems~\ref{bogruzsabilinearintro} and~\ref{algreglemmaintro}, along with other auxiliary results proved here, open up the possibility of extending the mentioned inverse theorem to arbitrary finite abelian groups. In fact, it is likely that the other arguments from the proof in~\cite{U4paper}, which are of a more combinatorial nature, could be combined with the results of this paper to understand Freiman bihomomorphisms\footnote{For finite abelian groups $G, H$ and $K$, and a subset $A \subseteq G \times H$, a \emph{Freiman bihomomorphism} is a map $\phi \colon A \to K$ with the property that the map $y \mapsto \phi(x, y)$ is a Freiman homomorphism for each $x \in G$ (with domain $\{y \in H \colon (x,y) \in A\}$) and the map $x \mapsto \phi(x, y)$ is a Freiman homomorphism for each $y \in H$ (again with an appropriate domain).} in general finite abelian groups. In the case of the $\mathbb{F}_p^n$, it is possible to deduce the inverse theorem for the $\|\cdot\|_{\mathsf{U}^4}$ norm from an inverse theorem for Freiman bihomomorphisms (see~\cite{U4paper} for $p \geq 5$ and a recent work of Tidor~\cite{Tidor} for the cases $p \in \{2,3\}$).\\

\noindent\textbf{`Linear algebra' in finite abelian groups.} Since we are no longer working with vector spaces, we need to be able to generalize several linear-algebraic facts to setting of finite abelian groups. Let us mention a few examples.\\ 

First, the vector space identity
\begin{equation}(U \cap V)^\perp = U^\perp + V^\perp\label{perpIdIntro}\end{equation}
plays a crucial role in the proof of the bilinear Bogolyubov argument. In the setting of general finite abelian groups, where $U$ and $V$ are no longer subspaces but are Bohr sets instead, we prove the following result.\\
\indent For a set $\Gamma = \{\gamma_1, \dots, \gamma_k\}$ of elements of an abelian group and a nonnegative integer $R$, we write $\langle \Gamma \rangle_R$ for the set of all linear combinations $\lambda_1\gamma_1 + \dots + \lambda_k \gamma_k$, where $\lambda_1, \dots, \lambda_k \in [-R, R] \subseteq \mathbb{Z}$. 

\begin{theorem}\label{dualsidentityintro}Let $G$ be a finite abelian group. Let $\Gamma_1, \Gamma_2 \subset \hat{G}$ and $\rho_1, \rho_2 \in (0,1)$. Then there is a positive integer $R \leq (2\rho_1^{-1})^{O(|\Gamma_1|)} + (2\rho_2^{-1})^{O(|\Gamma_2|)}$ such that
\[B(\langle \Gamma_1 \rangle_R \cap \langle \Gamma_2 \rangle_R, 1/4) \subseteq B(\Gamma_1, \rho_1) + B(\Gamma_2, \rho_2).\]
\end{theorem}

One way to think about the Bohr set $B(\Gamma; \rho)$ is to view it as an analogue of $\Gamma^\perp$ in finite abelian groups; by this analogy Theorem~\ref{dualsidentityintro} is a meaningful generalization of the linear-algebraic identity~\eqref{perpIdIntro}.\\

Another important notion is the linear independence of elements in a vector space. We may still talk about independent elements in finite abelian groups, but in our case we need a quantitative version of that notion. To that end, we prove the following theorem.

\begin{theorem}\label{latticesIntro}Let $G$ be a finite abelian group, let $a_1, \dots, a_k \in G$ and let $R$ be a positive integer. Suppose that $B \subset \langle a_1, \dots, a_k \rangle_R$ is a non-empty set. Then, there exist positive integers $\ell \leq O(k^2 (\log k + \log R))$, $S \leq O((2Rk)^{k+ 3})$ and elements $b_1, \dots, b_\ell \in B$ such that $B \subseteq \langle b_1, \dots, b_\ell\rangle_{S}$.\end{theorem}

To relate this to the linear-algebraic setting, we think of $ \langle a_1, \dots, a_k \rangle_R$ as a subspace $V$ of dimension $k$, and $B$ a subspace inside $V$. The usual conclusion then is that $B$ has a spanning set of size at most $k$. In the general setting we have to pay a modest price and allow a slightly larger spanning set, but we still have control over the coefficients of the required linear combinations.\\

As a final example, in order to prove Theorem~\ref{algreglemmaintro}, we need strong control over sizes of Bohr sets. To that end, we prove an approximate formula for the size of Bohr sets that are somewhat well-behaved.

\begin{proposition}[Size of Bohr sets]Let $k \in \mathbb{N}$ and let $\rho, \eta, \varepsilon > 0$. Then there exist a positive integer $K \leq O(k \varepsilon^{-1} \eta^{-1})$ and quantities $c_i \in \mathbb{D}$ for $i \in [-K, K]$ for which the following holds.\\
\indent Let $G$ be a finite abelian group. Suppose that $\gamma_1, \dots, \gamma_k \in \hat{G}$ and that 
\begin{equation*}|B(\gamma_1, \dots, \gamma_k; \rho + \eta) \setminus B(\gamma_1, \dots, \gamma_k; \rho)| \leq \varepsilon |G|.\end{equation*}
Then 
\[\Big||B(\gamma_1, \dots, \gamma_k; \rho)| - \sum_{a_1, \dots, a_k \in [-K, K]} \id(a_1 \gamma_1 + \dots + a_k \gamma_k = 0) c_{a_1} \dots c_{a_k} |G|\Big| \leq 2\varepsilon |G|.\]
\end{proposition}

For this proposition, the linear-algebraic analogue is the formula
\[|\langle u_1, \dots, u_k \rangle^\perp| \,= \Big|\Big\{(a_1, \dots, a_k) \in \mathbb{F}_p^k \colon a_1 u_1 + \dots + a_k u_k = 0\Big\}\Big|\, p^{-k}\,|\mathbb{F}_p^n| \,= |\mathbb{F}_p^n| / |\langle u_1, \dots, u_k \rangle|\]
which holds for any $u_1, \dots, u_k \in \mathbb{F}_p^n$, or, put more concisely, $|U^\perp| = p^n / |U|$ for all subspaces $U \leq \mathbb{F}_p^n$. Note also that we typically use the notion of regular Bohr sets when we need nice behaviour; in our case the weaker assumption above is sufficient. Finally, let us also remark that the constants $c_i$ do not depend on the ambient group $G$.\\

\noindent\textbf{Organization of the paper.} The first four sections are somewhat preliminary in the sense that we recall some basic properties and well-known results regarding the topic of each section, and then proceed to derive some new ones that are required in order to prove Theorem~\ref{bogruzsabilinearintro}. Section~\ref{secCP} is devoted mainly to coset progressions and Freiman homomorphisms,  Section~\ref{secFreiman} to Freiman's theorem, Section~\ref{secBohr} to Bohr sets and Section~\ref{secLattice} to lattices. Following those four sections, in Section~\ref{secQRBohr} we study quasirandomness properties of bilinear Bohr varieties and prove Theorem~\ref{algreglemmaintro}. Finally, in Section~\ref{secBilBA} we combine all these ingredients to complete the proof of Theorem~\ref{bogruzsabilinearintro}. There are also two appendices, the first one on the robust version of the Bogolyubov-Rusza lemma due to Schoen and Sisask~\cite{SchSisRob}, and the second one where we briefly recall the theory of quasirandom bipartite graphs.\\

\noindent\textbf{Acknowledgements.} This work was supported by the Serbian Ministry of Education, Science and Technological Development through Mathematical Institute of the Serbian Academy of Sciences and Arts. 

\section{Coset progressions and Freiman homomorphisms}\label{secCP}

\noindent \textbf{Coset progressions.} We begin the preliminary section by repeating the definition of coset progressions for completeness. A \emph{coset progression} in an abelian group $G$ is a set $C$ of the form $L_1 + \dots + L_r + H$, where $L_i$ are arithmetic progressions and $H \leq G$ is a subgroup. We say that $L_1 + \dots + L_r + H$ is a \emph{canonical form} of $C$. The number $r$ is known as the \emph{rank} of the coset progression. We say that a coset progression $C$ is \emph{proper} if $|C| = |L_1| \cdots |L_r||H|$, (thus all sums of $(r + 1)$-tuples are distinct), and we say that $C$ is \emph{symmetric} if all $L_i$ are symmetric, i.e. $L_i = -L_i$. Coset progressions were introduced by Green and Ruzsa in their generalization of Freiman's theorem to general abelian groups~\cite{greenRuzsaFreiman}, and, as their theorem confirms, these objects are a correct generalization of (cosets of) subgroups from an additive-combinatorial perspective.\\

We say that $A \subseteq B$ is a \emph{Freiman-subgroup} of $B$ if whenever $a, b \in A$ and $a - b \in B$ then $a - b \in A$. The next lemma show that Freiman-sbugroups of a coset progression $C$ are closely related to $C$.

\begin{lemma}\label{progsbgp}Let $G$ be a finite abelian group. Let $C = [-N_1, N_1] \cdot v_1 + \dots + [-N_d, N_d] \cdot v_d + H$ be a proper symmetric coset progression of rank $d$ in its canonical form. Let $A \subseteq C$ be a Freiman-subgroup of size $|A| \geq \alpha |C|$ in $G$. Then $A$ contains $[-M_1, M_1] \cdot \ell_1 v_1 + \dots + [-M_d, M_d] \cdot \ell_d v_d + H'$ for some positive integers $\ell_1, \dots, \ell_d \leq 20\alpha^{-1}$, $M_i = \lfloor N_i/\ell_i\rfloor$ and a subgroup $H' \leq H$ of size $|H'| \geq \alpha |H|$.\end{lemma}

\begin{proof}We claim that for each $i \in [d]$ the Freiman-subgroup $A$ contains an element of the form $\ell_i v_i$ for some $\ell_i \in [2\lceil \alpha^{-1} \rceil]$, unless $N_i \leq 10 \alpha^{-1}$. Without loss of generality $i = 1$. Suppose that $N_1 \geq 10 \alpha^{-1}$. Average over $x_2 \in [-N_2, N_2], \dots, x_d \in [-N_d, N_d]$ and $h \in H$ to find a choice of $x_2,\dots, x_k, h$ such that $|A \cap ([-N_1, N_1] \cdot v_1 + x_2 v_2 + \dots + x_d v_d + h)| \geq \alpha (2N_1 + 1)$. Let $A' = \{x_1 \in [-N_1, N_1] \colon x_1 v_1 + x_2 v_2 + \dots + x_d v_d + h \in A\}$ which has size at least $\alpha(2N_1 + 1)$.\\
\indent Let $L = 2\lceil \alpha^{-1} \rceil$. We split the interval $[-N_1, N_1]$ into disjoint intervals of the form $[-N_1 + j L, -N_1 + (j+1)L)$, except possibly the last one which looks like $[-N_1 + j L, N_1]$ but still has length at most $L$. The number of intervals is $\Big\lceil \frac{2N_1 + 1}{L}\Big\rceil$ which is smaller than the size of $A'$, so there are two indices $v_1 < v'_1 < v_1 + L$ in $A'$. Since $A$ is a Freiman-subgroup of $C$, we obtain the desired element.\\
\indent Take $\ell_i$ we obtained above for each $N_i \geq 10 \alpha^{-1}$. When $N_i < 10 \alpha^{-1}$, we simply put $\ell_i = \lfloor 20 \alpha^{-1} \rfloor$, which makes $M_i = 0$. To finish the proof we need to find an appropriate subgroup $H' \leq H$. Take $z \in [-N_1, N_1] \cdot v_1 + \dots + [-N_d, N_d] \cdot v_d$ such that $|A \cap z + H| \geq \alpha |z + H|$ and fix arbitrary $h_0 \in H \cap (A - z)$. For each $z + h \in A \cap z + H$ we have that $(z + h) - (z + h_0) = h - h_0 \in H \subseteq C$ and thus $h - h_0 \in A$. Thus $|A \cap H| \geq \alpha |H|$ and $A \cap H$ is a Freiman-subgroup $H$. Since $H$ is a group in the usual sense, so is $A \cap H$, so we set $H' = A \cap H$.\end{proof}

Later in this section, we shall prove an inverse theorem for approximate Freiman homomorphisms (Theorem~\ref{approxFreimanHom}). Similarly to the case when $G$ and $H$ are vector spaces over some field (for example, see~\cite{bogPaper}), this theorem will follow from Freiman's theorem and the Balog-Szemer\'edi-Gowers theorem. However, in the more general setting we need an additional algebraic ingredient, which is the following result.\\
\indent Recall that a module $P$ over a ring is \emph{projective} if for every surjective module homomorphism $f \colon N \to M$ and every module homomorphism $g \colon P \to M$, there exists a module homomorphism $h \colon P \to N$ such that $f \circ h = g$. In the case of the category of finite abelian groups, there are no projective objects. However, if we allow Freiman homomorphism on coset progressions, we may recover projectivity to some extent.

\begin{theorem}[Partial projectivity in abelian groups]\label{partialProjThm}Let $G$ and $H$ be finite abelian groups. Let $K \leq H$ be a subgroup. Suppose that $\phi \colon G \to H/K$ is a homomorphism. Let $s \geq 2$. Then there exist a proper coset progression $C \subseteq G$ of rank at most $\log_2 |K|$ and size $|C| \geq s^{- \log_2 |K|}|G|$ and a Freiman $s$-homomorphism $\psi \colon C \to H$ such that for all $x \in C$ we have $\phi(x) = \psi(x) + K$.\end{theorem}

We say that a sequence of elements $x_1, \dots, x_r$ in a finite abelian group $G$ is a \emph{basis} if the equality $\lambda_1 x_1 + \dots + \lambda_r x_r = 0$ for integers $\lambda_1, \dots, \lambda_r$ implies that $n_i | \lambda_i$ for each $i \in [r]$, where $n_i$ stands for the order of the element $x_i$, and we also have $n_1 \cdots n_r = |G|$, so elements $x_1, \dots, x_r$ span $G$. As in linear algebra, we may perform some linear manipulations to the sequence above without affecting the property of being a basis.

\begin{lemma}[Change of basis in abelian groups]\label{chbasisabgp}Suppose that $n_1 | n_2| \dots | n_r$. 
\begin{itemize}
\item[\textbf{(i)}] Let $i < j$. Then replacing $x_i$ by $x_i - \lambda \frac{n_j}{n_i} x_j$ preserves the property of being a basis, with orders of elements in the sequence unchanged.
\item[\textbf{(ii)}] Let $i > j$. Then replacing $x_i$ by $x_i - \lambda x_j$ preserves the property of being a basis, with orders of elements in the sequence unchanged.
\item[\textbf{(iii)}] Let $\lambda$ be an integer coprime to $n_i$. Replacing $x_i$ by $\lambda x_i$ preserves the property of being a basis, with orders of elements in the sequence unchanged.
\end{itemize}
\end{lemma}

\begin{proof}\textbf{Proof of (i).} Let $m$ be the order of $x_i - \lambda \frac{n_j}{n_i} x_j$. Then $mx_i - \lambda  m\frac{n_j}{n_i} x_j = 0$. Since $x_1, \dots, x_r$ is a basis, we have that $n_i | m$. On the other hand, $n_i\Big(x_i - \lambda \frac{n_j}{n_i} x_j\Big) = n_i x_i - \lambda n_j x_j = 0$, so $m = n_i$.\\
Since the orders of elements in the new basis are the same as in the old one, we just need to show independence. Let $\lambda_1, \dots, \lambda_r$ be integers such that 
\[\lambda_i \Big(x_i - \lambda \frac{n_j}{n_i} x_j\Big) + \sum_{\ell \in [r] \setminus \{i\}} \lambda_\ell x_\ell = 0.\]
The coefficients of elements $x_\ell$ for $\ell \not= j$ are $\lambda_\ell$, so $n_\ell | \lambda_\ell$. On the other hand, the element $x_j$ has $\lambda_j - \lambda_i \lambda \frac{n_j}{n_i}$ as its coefficient, so $n_j | \lambda_j - \lambda_i \lambda \frac{n_j}{n_i}$. But $n_i | \lambda_i$, so it follows that $n_j  | \lambda_j$, as desired.\\
\indent \textbf{Proof of (ii).} As in the proof of the property \textbf{(i)}, one easily sees that the order of $x_i - \lambda x_j$ is divisible by $n_i$. On the other hand, $n_i \lambda x_j = 0$ since $n_j | n_i$, so $n_i (x_i - \lambda x_j) = 0$.\\
Again, we just check independence. Let $\lambda_1, \dots, \lambda_r$ be integers such that 
\[\lambda_i \Big(x_i - \lambda x_j\Big) + \sum_{\ell \in [r] \setminus \{i\}} \lambda_\ell x_\ell = 0.\]
As before, $n_\ell | \lambda_\ell$ holds for all $\ell \not= j$. The element $x_j$ has $\lambda_j - \lambda_i \lambda$ as its coefficient, which has to be divisible by $n_j$. We know that $n_j | n_i | \lambda_i$ so it follows that $n_j | \lambda_j$, as desired.\\
\indent \textbf{Proof of (iii).} An easy check.\end{proof}

\begin{proof}[Proof of Theorem~\ref{partialProjThm}]Using the invariant factors decomposition, we know that $G$ is a direct sum $\mathbb{Z}/{n_1 \mathbb{Z}} \oplus \dots \oplus \mathbb{Z}/{n_r \mathbb{Z}}$ for some positive integers $n_1 | \dots | n_r$. In our terminology, this means that $G$ has a basis $x_1, \dots, x_r$ where $x_i$ has order $n_i$. For each $x_i$, pick an element $h_i \in H$ such that $\phi(x_i) = h_i + K$. Since $\phi$ is a homomorphism, we have that $K = \phi(n_ix_i) = n_i \phi(x_i) = n_ih_i + K$, so $n_i h_i \in K$. Furthermore, for any integers $\lambda_1, \dots, \lambda_r$ we have
\[\phi(\lambda_1 x_1 + \dots + \lambda_r x_r) = \lambda_1 h_1 + \dots + \lambda_r h_r + K.\]
Let us now perform the following procedure. We shall go through the sequence $x_1, \dots, x_r$ in the reverse order and define a new basis $y_1, \dots, y_r$ of $G$ with orders $n_1, \dots, n_r$, and new elements $k_1, \dots, k_r \in H$ so that $n_i k_i = 0$ instead of merely $n_ik_i \in K$ for all but at most $\log_2 |K|$ of indices $i \in [r]$ (for the rest of them, we only guarantee that $n_i k_i \in K$).\\
At each step, assuming that $y_r, \dots, y_{i + 1}$ and $k_r, \dots, k_{i+1}$ have been defined, we check whether the element $n_i h_i$ belongs to the subgroup $S$ of $K$ generated by elements $n_r k_r, \dots, n_{i + 1} k_{i+1}$. If $n_i h_i \notin S$, then we simply put $y_i = x_i$ and $k_i = h_i$. On the other hand, if $n_i h_i \in S$, that means that we can find integers $\lambda_{i+1}, \dots, \lambda_r$ such that
\[n_i h_i = \lambda_{i+1} n_{i+1} k_{i+1} + \dots + \lambda_r n_r k_r.\]
Let us put 
\begin{equation}y_i = x_i - \sum_{j \in [i+1, r]} \lambda_j \frac{n_j}{n_i} y_j\label{newbasisdefn1}\end{equation}
and 
\begin{equation}k_i = h_i - \sum_{j \in [i+1, r]} \lambda_j \frac{n_j}{n_i} k_j.\label{newbasisdefn2}\end{equation}

\begin{claim}\label{fromHtoKclaim}At the end of each step, the sequence $y_r, \dots, y_i, x_{i-1}, \dots, x_1$ is a basis with orders $n_r, \dots, n_1$. Moreover, for any choice of integers $\mu_1, \dots, \mu_r$ we have
\begin{equation}\phi\Big(\mu_1 x_1 + \dots + \mu_{i-1} x_{i-1} + \mu_i y_i + \dots + \mu_r y_r\Big) = \mu_1 h_1 + \dots + \mu_{i-1} h_{i-1} + \mu_i k_i + \dots + \mu_r k_r + K.\label{quotHommEq}\end{equation}
\end{claim}

\begin{proof}[Proof of the claim]We prove the claim by downwards induction on $i$. The base case is $i = r + 1$, where we consider the basis $x_r, x_{r-1}, \dots, x_1$, for which the claim holds. Suppose that the claim holds for some $i \geq 2$. By Lemma~\ref{chbasisabgp} we see that the order of $y_{i-1}$ is the same as the order of $x_{i-1}$ and that $y_r, \dots, y_{i-1}, x_{i-2}, \dots, x_1$ is still a basis.\\
\indent To prove the equality~\eqref{quotHommEq}, since $\phi \colon G \to H / K$ is a homomorphism, we just need to show that $\phi(y_i) = k_i + K$ for each $i \in [r]$, which we again do by downwards induction on $i$. But, we either have that $y_i = x_i$ and $k_i = h_i$, so we are done, or $y_i$ and $h_i$ are given by~\eqref{newbasisdefn1} and~\eqref{newbasisdefn2}, so 
\[\phi(y_i) = \phi\Big(x_i - \sum_{j \in [i+1, r]} \lambda_j \frac{n_j}{n_i} y_j\Big) = \phi(x_i) - \sum_{j \in [i+1, r]} \lambda_j \frac{n_j}{n_i} \phi(y_j) = h_i - \sum_{j \in [i+1, r]} \lambda_j \frac{n_j}{n_i} k_j + K = k_i + K,\] 
as desired.\end{proof}

Observe that whenever $y_i$ and $k_i$ are given by~\eqref{newbasisdefn1} and~\eqref{newbasisdefn2}, we have that $n_i k_i = 0$; namely
\[n_i k_i = n_i h_i - \sum_{j \in [i+1, r]} \lambda_j n_j k_j = 0,\]
by the choice of coefficients $\lambda_j$. On the other hand, any time we have $n_i k_i \not= 0$, that means that subgroup $\langle n_r k_r, \dots, n_i k_i\rangle \leq K$ is larger than before, and since it a group, it is at least twice as large as its predecessor. Thus, this can happen at most $\log_2 |K|$ times.\\

Let $I \subset [r]$ be the set of indices $i$ such that $n_i k_i \not= 0$. For $i \in I$ let $n'_i = \lceil n_i/s\rceil$. Consider the coset progression $C = \bigoplus_{i \in I} [0, n'_i - 1] \cdot y_i + G'$, where $G' = \bigoplus_{i \in [r]\setminus I} [0, n_i-1] \cdot y_i$, with $G'$ being a subgroup of $G$. The rank of $C$ is $|I| \leq \log_2 |K|$ and the size is $|C| = \Big(\prod_{i \in I} n'_i\Big)\Big(\prod_{i \in [r] \setminus I} n_i\Big) \geq s^{- \log_2|K|} |G|$, as desired. We now show that $C$ has the claimed property.\\

Let us define a map $\psi \colon C \to H$ by
\[\psi\Big(\sum_{i \in [r]} \lambda_i y_i\Big) = \sum_{i \in [r]} \lambda_i k_i\]
for $\lambda_i \in \mathbb{Z}$, with the restriction that $\lambda_i \in [0, n'_i-1]$ for $i \in I$. We claim that $\psi$ is a Freiman $s$-homomorphism.\\
\indent To that end, for each $i \in [r]$, let $\mu^1_i, \dots, \mu^s_i, \nu^1_i, \dots, \nu^s_i$ be integers, with the requirement that they belong to the interval $[0, n'_i-1]$ when $i \in I$. Suppose that for each $i$ we have $n_i | \sum_{j \in [s]} \mu^j_i - \sum_{j \in [s]} \nu^j_i$. However, when $i \in I$, using the fact that $n'_i = \lceil n_i/s\rceil$, we have the bound
\[\Big|\sum_{j \in [s]} \mu^j_i - \sum_{j \in [s]} \nu^j_i\Big| < n_i.\]
Thus, we have the exact equality $\sum_{j \in [s]} \mu^j_i = \sum_{j \in [s]} \nu^j_i$. Proceeding further, we have
\[\sum_{j \in [s]}\psi\Big(\sum_{i \in [r]} \mu^j_i y_i\Big) - \sum_{j \in [s]}\psi\Big(\sum_{i \in [r]} \nu^j_i y_i\Big) = \sum_{j \in [s]} \sum_{i \in [r]} \mu^j_i k_i - \sum_{j \in [s]} \sum_{i \in [r]} \nu^j_i k_i =  \sum_{i \in [r]} \Big(\sum_{j \in [s]} \mu^j_i - \sum_{j \in [s]} \nu^j_i\Big) k_i.\]
When $i \in I$, the equality $\sum_{j \in [s]} \mu^j_i = \sum_{j \in [s]} \nu^j_i$ implies that the contribution to the sum above from $i$\tss{th} term is 0. When $i \notin I$, we have $n_i | \sum_{j \in [s]} \mu^j_i - \sum_{j \in [s]} \nu^j_i$, but $n_i k_i = 0$, so the contribution is $0$ as well. Hence, $\psi$ is a Freiman $s$-homomorphism.\\

Finally, by Claim~\ref{fromHtoKclaim}, we see that for each $x \in C$ we have $\phi(x) = \psi(x) + K$, as required.\end{proof}

As a corollary, we deduce that somewhat injective Freiman homomorphisms on coset progressions have a large injective piece.

\begin{corollary}\label{almostinjectivitycor}Let $G$ and $H$ be finite abelian groups and let $C \subseteq G$ be a proper coset progression with canonical form $C = P + K$ where $P = v_0 + [0, N_1 - 1] \cdot v_1 + \dots + [0, N_r - 1] \cdot v_r$ and $K \leq G$. Let $\phi \colon C \to H$ be a Freiman homomorphism such that $|\phi(C)| \geq \alpha |C|$. Then there exist a further proper coset progression $D \subseteq K$ of rank $\log_2 \alpha^{-1}$ and size $\alpha^2 |K|$, a positive integer $m \leq (8r \alpha^{-1})^r$ and a partition $P = P_1 \cup \dots \cup P_m$ into proper progressions such that $\phi$ is injective on each proper coset progression of the form $k + P_i + D$ for $i \in [m]$ and $k \in K$.\end{corollary}

\begin{proof}Let us begin by defining an auxiliary map $\psi \colon K \to H$ by setting $\psi(k) = \phi(x + k) - \phi(x)$ for any $x \in C$. Note that $x \in C$ implies $x + k \in C$ and that $\phi$ being a Freiman homomorphism implies that $\psi$ is a homomorphism (in the usual group-theoretic sense). Furthermore, for each $x \in C$ we also have that $|\phi(x + K)| = |\psi(K)|$ so $|\psi(K)| \geq \alpha |K|$. In particular, if $S = \on{ker} \psi$, we have $|S| \leq \alpha^{-1}$.\\
\indent Now define a homomorphism $\nu \colon \psi(K) \to K/S$ by sending every $y \in \psi(K)$ to $x +S$ for $x \in K$ such that $\psi(x) = y$. Since $S$ is the kernel of $\psi$, we see that $\nu$ is well-defined. The fact that $\nu$ is a homomorphism follows from $\psi$ being a homomorphism. Apply Theorem~\ref{partialProjThm} to find a proper coset progression $\tilde{D} \subseteq \psi(K)$ of rank at most $\log_2 \alpha^{-1}$ and size at least $|\tilde{D}| \geq \alpha |\psi(K)|$ and a Freiman homomorphism $\theta \colon \tilde{D} \to K$ such that
\begin{equation}(\forall z \in \tilde{D})\,\, \nu(z) = \theta(z) + S.\label{nuthetayprop}\end{equation}
We claim that $\theta$ is injective. Suppose that there exist elements $z_1, z_2 \in \tilde{D}$ such that $\theta(z_1) = \theta(z_2)$. Since $z_1, z_2 \in \psi(K)$, there $x_1, x_2 \in K$ such that $\psi(x_1) = z_1$ and $\psi(x_2) = z_2$. By definition of $\nu$ we have $\nu(z_1) = x_1 + S$ and $\nu(z_2) = x_2 + S$. But,
\[x_1 + S = \nu(z_1) = \theta(z_1) + S = \theta(z_2) + S = \nu(z_2) = x_2 + S\]
so $x_1 - x_2 \in S$, giving $z_1 = \psi(x_1) = \psi(x_2) = z_2$, proving the injectivity of $\theta$.\\
\indent Write $D = \theta(\tilde{D})$. It follows that $D$ is a proper coset progression in $K$ and that $\theta^{-1} \colon D \to H$ is also a Freiman homomorphism. Furthermore, for each $x \in D$ we have $\theta_y^{-1}(x) \in \tilde{D}$ so by property~\eqref{nuthetayprop} we get $\nu(\theta^{-1}(x)) = \theta(\theta^{-1}(x)) + S = x + S$. The definition of $\nu$ gives $\psi(x) = \theta^{-1}(x)$. Since $\theta^{-1}$ is injective, it follows that $\psi$ is injective on $D$.\\
\indent For $i \in [r]$ set $N'_i = \lceil \alpha N_i/(4r)\rceil$. We claim that $\phi$ is injective on each coset progression
\begin{equation}\label{cosetprogressioninjectivity} v_0 + [a_1, a_1 + N'_1 - 1] \cdot v_1 + \dots + [a_r, a_r + N'_r - 1] \cdot v_r + D,\end{equation}
for any choice of $a_i \in [0, N_i - N'_i]$. Suppose on the contrary that $\phi$ was not injective on the set above. Then we obtain $a_1,b_1 \in [a_1, a_1 + N'_1 - 1], \dots, a_r, b_r \in [a_r, a_r + N'_r - 1], k, l \in D$ such that 
\[\phi\Big(v_0 + a_1 v_1 + \dots + a_r v_r + k\Big) = \phi\Big(v_0 + b_1 v_1 + \dots + b_r v_r + l\Big),\]
with arguments of $\phi$ being different elements. If we had $a_1 v_1 + \dots + a_r v_r = b_1 v_1 + \dots + b_r v_r$, that would have implied
\begin{align*}\psi(k) = &\phi\Big(v_0 + a_1 v_1 + \dots + a_r v_r + k\Big) - \phi\Big(v_0 + a_1 v_1 + \dots + a_r v_r\Big) \\
= &\phi\Big(v_0 + b_1 v_1 + \dots + b_r v_r + l\Big) - \phi\Big(v_0 + b_1 v_1 + \dots + b_r v_r\Big) = \psi(l)\end{align*}
which would be in contradiction with the fact that $\psi$ is injective on $D$. Thus $a_1 v_1 + \dots + a_r v_r \not= b_1 v_1 + \dots + b_r v_r$.\\ 
\indent Write $w = (a_1 - b_1) v_1 + \dots + (a_r - b_r) v_r + (k - l) \in G$. Since $\phi$ is a Freiman homomorphism, we have that $\phi(x) = \phi(x + w)$ whenever $x, x + w \in C$. Note that $|a_i - b_i| \leq N'_i - 1$. Observe that for every element $x \in C$, there is some $\lambda \in \mathbb{Z}$ such that $x + \lambda w$ is of the form $v_0 + \mu_1 v_1 + \dots + \mu_r v_r + k \in P + K$ where at least one $\mu_i$ satisfies $\mu_i \leq N'_i - 2$ or $\mu_i \geq N_1 - N'_i + 1$. In particular, $\phi(x + \lambda w) = \phi(x)$, so the total number of values attained by $\phi$ is at most
\[\sum_{i \in [r]} 2(N_i' - 1) \Big(\prod_{j \in [r] \setminus \{i\}} N_j\Big) |K| \leq 2r (\alpha / 4r) |C| < \alpha |C|\]
which is a contradiction.\\
\indent To finish the proof, partition $P + D$ into proper coset progressions of the form~\eqref{cosetprogressioninjectivity} by taking $a_i$ to be multiples of $N'_i$. The total number of coset progressions is at most $(8r \alpha^{-1})^r$, as claimed.\end{proof}

Another way to phrase the additional algebraic obstacle arising in the setting of general finite abelian groups that was not present in vector space setting is that whenever $U$ is a subspace of a vector space $V$ we may find a further subspace $W \leq V$ such that $U \oplus W = V$. Such decompositions are not available this time. Still, we may obtain a partial decomposition using a coset progression when the given subgroup is small.

\begin{corollary}Let $G$ be a finite abelian group with a subgroup $H$. Then there exists a coset progression $C$ of rank at most $\log_2 |H|$ and size at least $|H|^{-2}|G|$ such that $|C + H| = |C| |H|$.\end{corollary}

\begin{proof}Let $\pi \colon G \to G/ H$ be the natural projection sending $x$ to $x + H$. Since $|\pi(G)| = |H|^{-1}|G|$, Corollary~\ref{almostinjectivitycor} gives a coset progression $C$ of rank $\log_2 |H|$ and size at least $|H|^{-2}|G|$ on which $\pi$ is injective and so $|C + H| = |C| |H|$.\end{proof}

To see the need for coset progressions, look at $G = \mathbb{Z}/2^n\mathbb{Z}$ and $H = \{0, 2^{n-1}\}$. There is no decomposition using subgroups only as every non-trivial subgroup of $G$ contains $H$.\\

\section{Variants of Freiman's theorem}\label{secFreiman}

In this section we gather various results related to Freiman's theorem. Let us first state the Bogolyubov-Rusza lemma of Sanders~\cite{Sanders}, which plays a very important role in this work, as is the case with previous works on the bilinear Bogolyubov argument.

\begin{theorem}[Bogolyubov-Rusza lemma, Theorem 1.1~\cite{Sanders}]\label{bogRuzsa1}Suppose that $G$ is a finite abelian group and that $A, B \subseteq G$ are subsets such that $|A + B| \leq K \min \{|A|, |B|\}$. Then there exists a symmetric proper coset progression $C \subseteq A - A + B - B$ of rank at most $\log^{O(1)} (K)$ and size $|C| \geq \exp\Big(-\log^{O(1)}(K)\Big)|A+B|$.\end{theorem}

Furthermore, we need a robust version of the Bogolyubov-Rusza lemma due to Schoen and Sisask~\cite{SchSisRob}. The following formulation follows rather straightforwardly from their results and a quick deduction is presented in Appendix~\ref{robbogruzapp}. 

\begin{theorem}\label{robustBogRuzsa}Let $G$ be a finite abelian group. Let $A \subset G$ be a set such that $|A + A| \leq K |A|$. Then there exists a symmetric proper coset progression $C$ of rank at most $\log^{O(1)} (K)$ and size at least $|C| \geq \exp(-\log^{O(1)} (K))|A|$ such that for each $x \in C$ there are at least $\frac{1}{64K}|A|^3$ quadruples $(a_1, a_2, a_3, a_4) \in A^4$ such that $x = a_1 + a_2 - a_3 - a_4$.\end{theorem}

In particular, when $A$ is a dense subset of a coset progression $C$ of bounded rank then $A$ has small doubling so the theorem above applies and gives the following corollary. It is important that the actual size of $C$ does not play a role.

\begin{corollary}\label{robustBogRuzsaCP}Let $G$ be a finite abelian group and let $C$ be a coset progression of rank $d$. Suppose that $A \subseteq C$ has size $|A| \geq \alpha |C|$. Then there exists a symmetric proper coset progression $C'$ of rank at most $(d\log (\alpha^{-1})^{O(1)}$ and size at least $|C'| \geq  \exp\Big(-(d\log (\alpha^{-1}))^{O(1)}\Big) |C|$ such that for each $x \in C'$ there are at least $\frac{\alpha}{2^{d + 6}}$ quadruples $(a_1, a_2, a_3, a_4) \in A^4$ such that $x = a_1 + a_2 - a_3 - a_4$.\end{corollary}

\begin{proof}The doubling of $A$ is at most $2^d \alpha^{-1}$. Apply Theorem~\ref{robustBogRuzsa}.\end{proof}

Next, we need an inverse theorem for approximate Freiman homomorphisms. Recall that a quadruple $(a,b,c,d)$ of elements of an abelian group is said to be \emph{additive} if $a + b = c + d$ holds, and that it is \emph{respected} by a map $\phi$ if the codomain of $\phi$ is an abelian group, the elements $a,b,c,d$ belong to the domain of $\phi$ and $\phi(a) + \phi(b) = \phi(c) + \phi(d)$ holds.

\begin{theorem}\label{approxFreimanHom}Let $G$ and $H$ be finite abelian groups, let $A \subset G$ be a subset and let $\phi \colon A \to H$ be a function. Suppose that the number of additive quadruples respected by $\phi$ is at least $\delta |G|^3$. Then there exist a proper coset progression $Q \subseteq G$ of rank at most $\log^{O(1)} (\delta^{-1})$ and a Freiman homomorphism $\Phi \colon Q \to H$ such that $\Phi(x) = \phi(x)$ holds for at least $\exp(-\log^{O(1)} (\delta^{-1})) |G|$ elements $x \in A \cap Q$.\end{theorem}

Such a result for cyclic groups is due to Gowers~\cite{Tim4ap}, while a similar statement for general finite abelian groups (with weaker bounds than the ones above) can be extracted from the work of Green and Tao on the inverse question for $\|\cdot\|_{\mathsf{U}^3}$ norm~\cite{StrongU3}. However, since we need nearly optimal bounds and must use Sanders's Bogolyubov-Ruzsa lemma rather than Freiman's theorem, we need to give a proof.

\begin{proof}We consider the graph $\Gamma = \{(a, \phi(a)) \colon a \in A\}$ of the mapping $\phi$ which is a subset of the abelian group $G \times H$. Every additive quadruple $(a,b,c,d) \in A^4$ which is respected by $\phi$ produces an additive quadruple $\Big((a, \phi(a)), (b, \phi(b)), (c, \phi(c)), (d, \phi(d))\Big)$ in $\Gamma$. Thus, there are at least $\delta |\Gamma|^3$ additive quadruples in $\Gamma$. Applying Balog-Szemer\'edi-Gowers theorem (\cite{BalogSzemeredi},~\cite{Tim4ap}; we use the version stated as Theorem 2.29 in~\cite{TaoVuBook}), we may find subsets $\Gamma', \Gamma'' \subseteq \Gamma$ of size $|\Gamma'|, |\Gamma''| \geq \Omega(\delta^{O(1)}|\Gamma|)$ with $|\Gamma' + \Gamma''| \leq O(\delta^{-O(1)} |\Gamma|)$. By Pl\"{u}nnecke-Ruzsa inequality~(Theorem 6.29 in~\cite{TaoVuBook}, see also~\cite{RuzsaBook} and an alternative proof due to Petridis~\cite{Petridis}), we see hat $|\Gamma' + \Gamma'| \leq O(\delta^{-O(1)} |\Gamma'|)$. By Theorem~\ref{robustBogRuzsa} we find a proper coset progression $C' \subseteq 2\Gamma' - 2\Gamma'$ of size $|C'| \geq \exp(-\log^{O(1)} (\delta^{-1}))|\Gamma|$ and rank at most $\log^{O(1)} (\delta^{-1})$ such that for each $x \in C'$ there exist at least $\Omega(\delta^{O(1)})|\Gamma|^3$ quadruples $(a,b,c,d) \in {\Gamma'}^4$ such that $a + b - c - d = x$. Averaging over $b,c,d \in \Gamma'$ we obtain an element $t \in G \times H$ such that $|\Gamma' \cap C' - t| \geq \exp(-\log^{O(1)} (\delta^{-1})) |\Gamma|$. Note also that $|C'| \leq |2\Gamma' - 2\Gamma'| \leq O(\delta^{-O(1)} |\Gamma|)$, by another application of the Pl\"{u}nnecke-Ruzsa inequality. Returning to $\Gamma$ which contains the set $\Gamma'$, we conclude that there is proper coset progression $C = C' - t$ of rank $\log^{O(1)} (\delta^{-1})$  and $|\Gamma \cap C| = \alpha (|\Gamma| + |C|)$ for some $\alpha \geq \exp(-\log^{O(1)} (\delta^{-1}))$.\\
\indent Let $\pi \colon G \times H \to G$ be the projection onto the first coordinate. Using Corollary~\ref{almostinjectivitycor}, we may pass to a proper coset progression $\tilde{C} \subseteq C$ on which $\pi$ is injective and which satisfies $|\Gamma \cap \tilde{C}| \geq \exp(-\log^{O(1)} (\delta^{-1})) (|\Gamma| + |\tilde{C}|)$. Let $Q = \pi(\tilde{C})$, which is a proper coset progession in $G$. Define the map $\Phi \colon Q \to H$ by setting $\Phi(x)$ to be the unique element of $H$ such that $(x, \Phi(x)) \in \tilde{C}$. Explicitly, we know that $\tilde{C} = (x_0, y_0) + [M_1] \cdot (x_1, y_1) + \dots + [M_r] \cdot (x_r, y_r) + K'$ for a subgroup $K' \leq G \times H$. If $K'$ is isomorphic to $\mathbb{Z}/L_1\mathbb{Z} \oplus \dots \oplus \mathbb{Z}/L_s\mathbb{Z}$ for suitable $L_1, \dots, L_s$, we take arbitrary generator $(g_i, h_i)$ of the summand $\mathbb{Z}/L_i\mathbb{Z}$. Thus
\[\tilde{C} = (x_0, y_0) + [M_1] \cdot (x_1, y_1) + \dots + [M_r] \cdot (x_r, y_r) + [L_1] \cdot (g_1, h_1) + \dots + [L_s] \cdot (g_s, h_s),\]
which we use to define
\[\Phi(x_ 0 + \lambda_1 x_1 + \dots + \lambda_r x_r + \mu_1 g_1 + \dots + \mu_s g_s) = y_0 + \lambda_1 y_1 + \dots + \lambda_r y_r + \mu_1 h_1 + \dots + \mu_s h_s,\]
for $\lambda_i \in [M_i]$ and $\mu_i \in [L_i]$. This is a Freiman homomorphism on $Q$ and $\Phi(x) = \phi(x)$ holds for at least $\exp(-\log^{O(1)} (\delta^{-1})) |G|$ elements $x \in A \cap Q$, as desired.\end{proof}

Next, we use Freiman's theorem to prove that intersections of coset progressions are essentially still coset progressions. It is likely that this fact can be proved without using Freiman's theorem, but given the negligible loss in efficiency we opted to rely on it in order to simplify the proof. 

\begin{proposition}\label{cprogintersection}Let $C_1, \dots, C_r \subseteq G$ be proper coset progressions of rank at most $d$. Suppose that we are given a set $X \subseteq C_1 \cap \dots \cap C_r$ of size $|X| \geq \delta |G|$. Then there is a proper coset progression $C \subseteq C_1 \cap \dots \cap C_r$ of rank at most $(2dr \log \delta^{-1})^{O(1)}$ such that $|C \cap X| \geq \exp(-(2dr \log \delta^{-1})^{O(1)})|G|$.\end{proposition}

\begin{proof} Let $C_i = a_i + [0, N^{(i)}_1 - 1] \cdot v^{(i)}_1 + \dots + [0, N^{(i)}_d - 1] \cdot v^{(i)}_d + H_i$ be $i$\tss{th} proper coset progression its canonical form for each $i \in [r]$. For $i \in [r]$, $j \in [d]$, let $M^{(i)}_j = \Big\lceil \frac{\delta N^{(i)}_j}{100 d r}\Big\rceil$ and let $T^{(i)}_j = \lceil N^{(i)}_j/M^{(i)}_j\rceil$. Consider the following sub-coset progressions for each $i \in [r]$. For each $\ell \in [d]$, let $j_\ell$ range over $[4, T^{(i)}_\ell - 4]$ when $T^{(i)}_\ell \geq 50 \delta^{-1} dr$. On the other hand, when $T^{(i)}_\ell < 50 \delta^{-1} dr$, we redefine $M^{(i)}_\ell$ to be 1, and also set $T^{(i)}_\ell$ to $N^{(i)}_\ell$ instead, which satisfies $N^{(i)}_\ell \leq 100\delta^{-1} dr$ in this case, and we let $j_\ell$ range over $[0, T^{(i)}_\ell]$. Let $\mathcal{J}_i$ be the set of all $d$-tuples of indices $(j_1, \dots, j_d)$ that arise in this way. Note also that in either case $T^{(i)}_\ell \leq 200 \delta^{-1}dr$.\\
\indent Given $(j_1, \dots, j_d) \in \mathcal{J}_i$, we define the coset progression 
\begin{equation}\label{intercprogeq}C_{i, j_1, \dots, j_d} = a_i + j_1 M^{(i)}_1 v^{(i)}_1 + \dots + j_d M^{(i)}_d v^{(i)}_d + [0, M^{(i)}_1 - 1] \cdot v^{(i)}_1 + \dots + [0, M^{(i)}_d - 1] \cdot v^{(i)}_d + H_i.\end{equation}
For fixed $i$, the coset progressions $C_{i, j_1, \dots, j_d}$ with different indices $(j_1, \dots, j_d)$ are disjoint and they almost cover $C_i$ in the sense that
\[\Big|C_i \setminus \Big(\bigcup_{(j_1, \dots, j_d) \in \mathcal{J}_i} C_{i, j_1, \dots, j_d}\Big)\Big| \leq \frac{\delta}{2r} |G|.\]

Thus, we obtain
\[\Big|X \cap \bigcap_{i \in [r]}\Big(\bigcup_{(j_1, \dots, j_r)\in \mathcal{J}_i} C_{i, j_1, \dots, j_r}\Big)\Big| \geq \frac{\delta}{2}|G|.\]

By averaging, we may find a coset progression $C'_i$ of the form $C_{i, j_1, \dots, j_r}$ for each $i \in [r]$ such that 
\[|X \cap C'_1 \cap \dots \cap C'_r| \geq \frac{1}{2}\,\frac{\delta^{dr + 1}}{(200rd)^{rd}}|G|.\]

Write $Y = X \cap C'_1 \cap \dots \cap C'_r$ and $K = 2\delta^{-(dr + 1)}(200rd)^{rd}$, and let $y_0 \in Y$ be an arbitrary element. We claim that $y_0 + 4Y - 4Y \subset C_i$ for each $i \in [r]$ and therefore $y_0 + 4Y - 4Y \subset C_1 \cap \dots \cap C_r$. To that end, fix an index $i \in [r]$. We have that $Y \subseteq C'_i$ and $C'_i$ is of the form~\eqref{intercprogeq} and in particular
\[4Y - 4Y \subseteq 4C'_i - 4C'_i \subseteq [-4(M^{(i)}_1 - 1), 4(M^{(i)}_1 - 1)] \cdot v^{(i)}_1 + \dots + [-4(M^{(i)}_d - 1), 4(M^{(i)}_d - 1)] \cdot v^{(i)}_d + H_i.\]
By the way we defined the indices set $\mathcal{J}_i$ we conclude that $y_0 + 4Y - 4Y$ is indeed a subset of $C_i$, as claimed. Hence, given any $y_1, \dots, y_8 \in Y$ we have that $y_0 + y_1 + y_2 + y_3 + y_4 - y_5 - y_6 - y_7 - y_8 \in C_i$.\\

Apply Theorem~\ref{robustBogRuzsa} to find a proper coset progression $C$ of rank at most $(2\log K)^{O(1)}$ and size $|C| \geq \exp\Big(-(2\log K)^{O(1)}\Big)|G|$ such that for each $x \in C$ there are at least $\frac{1}{64K}|G|^3$ quadruples $(y_1, y_2, y_3, y_4) \in Y^4$ such that $x = y_1 + y_2 - y_3 - y_4$. By averaging, we may find $y_2, y_3, y_4 \in Y$ such that 
\[|C - y_2 + y_3 + y_4 \cap Y| \geq \exp\Big(-(2\log K)^{O(1)}\Big)|G|.\]
Finally, we have
\[C - y_2 + y_3 + y_4 \subseteq 2Y - 2Y - Y + 2Y = 4Y - 3Y \subseteq y_0 + 4Y - 4Y \subseteq C_1 \cap \dots \cap C_r,\]
so we can take $C - y_2 + y_3 + y_4$ to be the desired coset progression.\end{proof}

\noindent\textbf{Linear maps covering lemma.} The vector space version of the following fact plays an important role in the previous proofs of the bilinear Bogolyubov argument. It essentially says that if we are given sets of bounded size that have many linear relationships between them, then we can hit significant proportion of the sets by images of a single linear function. The formulation below is a generalization of Lemma 2.1 in~\cite{HosseiniLovett}.

\begin{lemma}\label{linearcoverlemma}Let $G$ and $H$ be finite abelian groups and let $A \subseteq G$. For each $a \in A$, let $U_a \subseteq H$ be a set of size at most $K$. Let $L_1 \colon C_1 \to H, \dots, L_m \colon C_m \to H$ be Freiman homomorphisms with domains $C_1, \dots, C_m \subseteq G$ which are proper coset progressions. Assume that one of these maps is the zero map $g \mapsto 0$ defined for all $g \in G$. Suppose that for each $a \in A$ we have chosen a subset of these maps denoted by $\mathcal{L}_a \subseteq \{L_1, \dots, L_m\}$ so that the zero map belongs to all subsets, i.e\ $0 \in \mathcal{L}_a$. Write $\mathcal{L}_a(a) = \{L(a) \colon L \in \mathcal{L}_a\}$. Suppose that the properties $z,w, y+z, y + w \in A$ and
\[(U_{y + z} - U_z) \cap (U_{y + w} - U_w) \not\subseteq (\mathcal{L}_{y + z}(y + z) - \mathcal{L}_z(z)) + (\mathcal{L}_{y + w}(y + w) - \mathcal{L}_w(w))\]
hold for at least $\varepsilon |G|^3$ choices of $y,z,w \in G$. Then, provided $|G| \geq 100 K^4 \varepsilon^{-1}$, there exists a Freiman homomorphism $L \colon C \to H$ defined on a proper coset progression $C$ of rank at most $\log^{O(1)}(\varepsilon^{-1} K)$ such that $L(y) \in U_y \setminus \mathcal{L}_y(y)$ holds for at least $\exp\Big(-\log^{O(1)}(\varepsilon^{-1} K)\Big)|G|$ elements $y \in C \cap A$.
\end{lemma}

We do not use the facts that $C_1, \dots, C_m$ are coset progressions and that $L_1, \dots, L_m$ are Freiman homomorphisms in the proof of the lemma, but we opted for this formulation as the relevant maps will have this structure when the lemma is applied.

\begin{proof}Consider a choice of $y,z,w \in G$ such that $z,w, y+z, y + w \in A$ and
\[(U_{y + z} - U_z) \cap (U_{y + w} - U_w) \not\subseteq (\mathcal{L}_{y + z}(y + z) - \mathcal{L}_z(z)) + (\mathcal{L}_{y + w}(y + w) - \mathcal{L}_w(w)).\]
Thus, there are elements $a \in U_{y + z}, b \in U_z, c \in U_{y + w}, d \in U_w$ such that $a - b = c - d \notin (\mathcal{L}_{y + z}(y + z) - \mathcal{L}_z(z)) + (\mathcal{L}_{y + w}(y + w) - \mathcal{L}_w(w))$. Since $0 \in \mathcal{L}_t(t)$ for all $t$, we see that we cannot have $a \in \mathcal{L}_{y + z}(y + z)$ and $b \in \mathcal{L}_z(z)$ simultaneously. Similarly, we cannot have $c \in \mathcal{L}_{y + w}(y + w)$ and $d \in \mathcal{L}_w(w)$ simultaneously. Thus
\[\Big(a \notin \mathcal{L}_{y + z}(y + z)\,\text{OR}\,b \notin \mathcal{L}_z(z)\Big)\,\text{AND}\,\Big(c \notin \mathcal{L}_{y + w}(y + w)\,\text{OR}\,d \notin \mathcal{L}_w(w)\Big).\] 
This implies one of the four possibilities:
\begin{itemize}
\item $a \notin \mathcal{L}_{y + z}(y + z)$\,\text{AND}\,$c \notin \mathcal{L}_{y + w}(y + w)$,
\item $a \notin \mathcal{L}_{y + z}(y + z)$\,\text{AND}\,$d \notin \mathcal{L}_w(w)$,
\item $b \notin \mathcal{L}_z(z)$\,\text{AND}\,$c \notin \mathcal{L}_{y + w}(y + w)$,
\item $b \notin \mathcal{L}_z(z)$\,\text{AND}\,$d \notin \mathcal{L}_w(w)$.
\end{itemize}
Without loss of generality, the last possibility occurs most frequently. Therefore, we obtain a set $T \subseteq G^3$ of size $|T| \geq \frac{\varepsilon}{4}|G|^3$ consisting of triples $(y,z,w) \in G^3$ such that $z,w, y+z, y + w \in A$,
\[(U_{y + z} - U_z) \cap (U_{y + w} - U_w) \not\subseteq (\mathcal{L}_{y + z}(y + z) - \mathcal{L}_z(z)) + (\mathcal{L}_{y + w}(y + w) - \mathcal{L}_w(w)),\]
and there are elements $a = a(y,z,w) \in U_{y + z}, b = b(y,z,w) \in U_z \setminus \mathcal{L}_z(z), c = c(y,z,w) \in U_{y + w}, d = d(y,z,w) \in U_w \setminus \mathcal{L}_w(w)$ that satisfy $a - b = c - d$.\\

For each $y \in A$, pick a random value $f(y) \in U_y$ uniformly and independently. For each $(y,z,w) \in T$ such that the elements $z,w,y + z, y + w$ are distinct we have
\[\mathbb{P}\Big(f(y + z) = a(y,z,w), f(z) = b(y,z,w), f(y + w) = c(y,z,w), f(w) = d(y,z,w)\Big) \geq K^{-4}.\]
By linearity of expectation, there is a choice of $f$ such that the equalities $f(y + z) = a(y,z,w)$, $f(z) = b(y,z,w)$, $f(y + w) = c(y,z,w)$ and $f(w) = d(y,z,w)$ hold for at least $\frac{\varepsilon}{4K^4}|G|^3 - 6|G|^2$ choices of $(y,z,w) \in T$. Let $T'$ be the subset of such triples. Let $A' = \{x \in A \colon f(x) \in U_x \setminus \mathcal{L}_x(x)\}$. Recalling that $b(y,z,w) \in U_z \setminus \mathcal{L}_z(z)$ and that $d(y,z,w) \in U_w \setminus \mathcal{L}_w(w)$ hold for all triples $(y,z,w) \in T$, we conclude that $(y,z,w) \in T'$ implies $z,w \in A'$.\\
\indent Recalling also that $a(y,z,w) - b(y,z,w) = c(y,z,w) - d(y,z,w)$ holds for all $(y,z,w) \in T$, it follows that there are at least $\frac{\varepsilon}{4K^4}|G|^3 - 6|G|^2 \geq \frac{\varepsilon}{8K^4}|G|^3$ quadruples $(a', b', c', d') \in A \times A' \times A \times A'$ such that $a' - b' = c' - d'$ and $f(a') - f(b') = f(c') - f(d')$, which can be rewritten as $a' - c' = b' - d'$ and $f(a') - f(c') = f(b') - f(d')$. (We used the assumption $|G| \geq 100 K^4 \varepsilon^{-1}$ to simplify the inequality above.) By Cauchy-Scwarz inequality, it follows that the restriction $f|_{A'}$ respects at least $\Omega\Big(\varepsilon^2K^{-8}|G|^3\Big)$ additive quadruples.\\
\indent Apply Theorem~\ref{approxFreimanHom} to obtain a proper coset progression $C$ of rank $\log^{O(1)}(\varepsilon^{-1} K)$ and a Freiman homomorphism $L \colon C \to H$ such that $L(x) = f(x)$ holds for at least $\exp\Big(-\log^{O(1)}(\varepsilon^{-1} K)\Big)|G|$ of $x \in C \cap A'$. In particular, such an element $x$ also satisfies $L(x) \in U_x \setminus \mathcal{L}_x(x)$, as desired.\end{proof}

Applying the lemma iteratively gives the following corollary.

\begin{corollary}\label{linearCoverCor}Let $G$ and $H$ be finite abelian groups. Let $A \subset G$ be a subset of density $\alpha$. For each $a \in A$, let $U_a \subseteq H$ be a set of size at most $K$ containing 0. Assume $|G| \geq 2000 K^4 \alpha^{-4}$. Then there exist a positive integer $m \leq \exp\Big(\log^{O(1)} (\alpha^{-1}K)\Big)$, proper coset progressions $C_1, \dots, C_m$ of rank at most $\log^{O(1)} (\alpha^{-1}K)$ and Freiman homomorphisms $L_1 \colon C_1 \to H, \dots, L_m \colon C_m \to H$ such that, writing $U'_a = \{L_i(a) \colon i \in [m], a \in C_i\} \cap U_a$ for each $a \in A$, we have $y + z, z, y + w, w \in A$ and 
\[(U_{y + z} - U_z) \cap (U_{y + w} - U_w) \subseteq (U'_{y + z} - U'_z) + (U'_{y + w} - U'_w)\]
for at least $\frac{\alpha^4}{2}|G|^3$ of $y,z,w \in G$.
\end{corollary}

\begin{proof}We begin with defining $L_1 \colon G \to H$ to be the zero map $x \mapsto 0$. We iteratively find further maps $L_2, L_3, \dots$ as long as the conclusion is not satisfied.\\
\indent Suppose therefore that so far we have found Freiman homomorphisms $L_1 \colon C_1 \to H, \dots, L_s \colon C_s \to H$ such that for at least $\frac{\alpha^4}{2}|G|^3$ triples $(y,z,w) \in G^3$ we have $y + z, z, y + w, w \in A$ but
\[(U_{y + z} - U_z) \cap (U_{y + w} - U_w) \not\subseteq (U'_{y + z} - U'_z) + (U'_{y + w} - U'_w),\]
where we used the basic fact that the number of additive quadruples in $A$ is at least $\alpha^4|G|^3$. We may use Lemma~\ref{linearcoverlemma}, where we set $\mathcal{L}_a$ to consist of those maps $L_i$ such that $a \in C_i$ and $L_i(a) \in U_a$. Note that since $0 \in U_a$ we have $L_1 \in \mathcal{L}_a$ as required by the lemma and that $|G| \geq 2000 K^4 \alpha^{-4}$ ensure the requirement on the size of $G$ is met. The lemma produces a proper coset progression $C_{s+1}$ of rank at most $\log^{O(1)}(\alpha^{-1} K)$ and a Freiman homomorphism $L_{s+1} \colon C_{s+1} \to H$ such that $L_{s+1}(y) \in U_y \setminus U'_y$ for at least $\exp\Big(-\log^{O(1)}(\alpha^{-1} K)\Big)|G|$ of $y \in C_{s+1} \cap A$. The procedure necessarily terminates after at most $\exp\Big(\log^{O(1)}(\alpha^{-1} K)\Big)$ steps as each time we cover at least $\exp\Big(-\log^{O(1)}(\alpha^{-1} K)\Big)|G|$ pairs of the form $(y,h)$ for some $y \in A$ and $h \in  U_y \setminus U'_y$ by pairs of the form $(y, L_{s+1}(y))$ and the total number of pairs that we start with is at most $\sum_{y \in A} |U_y| \leq K|G|$.\end{proof}

\section{Behaviour of Bohr sets}\label{secBohr}

In this section we recall the definition of Bohr sets, study their properties motivated by linear algebra and in particular derive an approximate formula that controls their size.\\

\noindent \textbf{Bohr sets.} Let $G$ be a finite abelian group. By the fundamental theorem for finitely generated abelian groups, we know that $G$ can without loss of generality be taken to be of the form $G = \mathbb{Z}/q_1 \mathbb{Z}\, \oplus\, \mathbb{Z}/q_2 \mathbb{Z} \,\oplus\, \dots \,\oplus\,\mathbb{Z}/q_d \mathbb{Z}$ for some natural numbers $q_1, \dots, q_d$ such that $q_1 | q_2 | \dots | q_k$. The dual group of $G$, denoted by $\hat{G}$, consists of characters of $G$, which we view as homomorphisms between abelian groups $G$ and $\mathbb{T} = \mathbb{R}/\mathbb{Z}$. We may put $\hat{G}$ in the following, explicit form. The dual group $\hat{G}$ has the structure $\mathbb{Z}/q_1 \mathbb{Z}\, \oplus\, \mathbb{Z}/q_2 \mathbb{Z} \,\oplus\, \dots \,\oplus\,\mathbb{Z}/q_d \mathbb{Z}$ as well. To see this, notice that for each $\chi \in \hat{G}$, there exist unique elements $\chi_i \in \mathbb{Z}/q_i\mathbb{Z}$ for each $i \in [d]$ such that 
\[\chi(x) = \sum_{i \in [d]} \frac{|\chi_i x_i|_{q_i}}{q_i} + \mathbb{Z},\]
where $|\cdot|_q \colon \mathbb{Z}/ q\mathbb{Z} \to \{0, 1, \dots, q-1\} \subseteq \mathbb{Z}$ is a map which maps each residue to the unique integer among $\{0,1, \dots, q-1\}$ that projects to it inside $\mathbb{Z}/ q\mathbb{Z}$.\\
\indent Write $\on{e}(t) = \exp(2 \pi i t)$ for $t \in \mathbb{R}$. For $x \in \mathbb{R}/\mathbb{Z}$ let $\tn{x}$ be the element $d \in [0, 1/2]$ such that $x \in \{-d,d\} + \mathbb{Z}$, i.e.\ the distance from 0. Given a set of characters $\Gamma \subseteq \hat{G}$ and a positive real $\rho$ we define the \emph{Bohr set} with \emph{frequency set} $\Gamma$ and \emph{width} $\rho$ as $B(\Gamma; \rho) = \{x \in G \colon \max_{\chi \in \Gamma}\tn{\chi(x)} \leq \rho\}$.\\  

The most basic fact about Bohr sets is that they are necessarily dense and of small doubling.

\begin{lemma}[Lemma 4.20 in~\cite{TaoVuBook}]\label{basicbohrsizel}For $\Gamma \subseteq \hat{G}$ and $\rho > 0$ we have
\[|B(\Gamma; \rho)| \geq \rho^{|\Gamma|}|G|\]
and
\[|B(\Gamma; 2\rho)| \leq 4^{|\Gamma|}|B(\Gamma; \rho)|.\]
\end{lemma}

The following elementary lemma shows that very dense subsets of Bohr sets contain large Bohr sets in their difference set.

\begin{lemma}\label{almostfullBohr} Let $A \subset B(\Gamma; \rho)$ be a set of size at least $(1 - 4^{-k-1})|B(\Gamma; \rho)|$ where $k = |\Gamma|$. Then $A - A \supseteq B(\Gamma; \rho/2)$.\end{lemma}

\begin{proof}Let $x \in B(\Gamma; \rho/2)$ be arbitrary. Then $x + B(\Gamma; \rho/2) \subseteq B(\Gamma; \rho)$ so we have
\[|A \cap x + B(\Gamma; \rho/2)| \geq |B(\Gamma; \rho/2)| - |B(\Gamma; \rho) \setminus A| \geq |B(\Gamma; \rho/2)| - 4^{-k-1}|B(\Gamma; \rho)| \geq \frac{3}{4} |B(\Gamma; \rho/2)|,\]
where we used Lemma~\ref{basicbohrsizel} in the last inequality.\\
\indent Specializing to $x = 0$, we see that $|A \cap B(\Gamma; \rho/2)| \geq \frac{3}{4} |B(\Gamma; \rho/2)|$. Thus, there is $a \in B(\Gamma; \rho/2) \cap A \cap (A - x)$, so we obtain $a, a+ x \in A$, as required.\end{proof}

\noindent \textbf{Fourier analysis on Bohr sets.} The following proposition determines the behaviour of size of Bohr sets. 

\begin{proposition}[Size of Bohr sets]\label{bohrsizeformula} Let $k \in \mathbb{N}$ and let $\rho, \eta, \varepsilon > 0$. Then, there exist a positive integer $K \leq O(k \varepsilon^{-1} \eta^{-1})$ and quantities $c_i \in \mathbb{D}$ for $i \in [-K, K]$ for which the following holds.\\
\indent Let $G$ be a finite abelian group. Suppose that $\gamma_1, \dots, \gamma_k \in \hat{G}$ and that 
\begin{equation}\label{wregcondition}|B(\gamma_1, \dots, \gamma_k; \rho + \eta) \setminus B(\gamma_1, \dots, \gamma_k; \rho)| \leq \varepsilon |G|.\end{equation}
Then 
\[\Big||B(\gamma_1, \dots, \gamma_k; \rho)| - \sum_{a_1, \dots, a_k \in [-K, K]} \id(a_1 \gamma_1 + \dots + a_k \gamma_k = 0) c_{a_1} \dots c_{a_k} |G|\Big| \leq 2\varepsilon |G|.\]
\end{proposition}

\noindent\textbf{Remark.} The proof gives explicit coefficients
\[c_a = \eta^{-1}\frac{\Big(\on{e}(a (\rho + \eta/2)) - \on{e}(-a (\rho + \eta/2))\Big)\Big(\on{e}(a \eta/2) - \on{e}(-a \eta/2)\Big)}{4 \pi^2 a^2},\]
which do not depend on the ambient abelian group $G$, but we shall not need this fact.\\

We refer to property~\eqref{wregcondition} as \emph{weak-regularity}. We opted for this notion instead of that of a regular Bohr set, which is usually used, since the weak-regularity condition is precisely what we need and it is easy to obtain an abundance of parameters $\eta$ and $\rho$ that satisfy~\eqref{wregcondition}.\\

\begin{proof}Let $I, J \subseteq \mathbb{T}$ be the intervals $I = [-(\rho + \eta/2),  (\rho + \eta/2)]$ and $J = [-\eta N / 2, \eta N / 2]$. The function $f(x) = \eta^{-1}I * J(x)$ satisfies $f(x) = 1$ when $x \in [-\rho,\rho]$, $f(x) = \frac{\rho + \eta - |x|_{\mathbb{T}}}{\eta} \in [0,1]$ when $|x|_{\mathbb{T}} \in [\rho, \rho + \eta]$ and $f(x) = 0$ when $x \notin [-(\rho + \eta), \rho + \eta]$. It follows that
\[|G|^{-1} |B(\gamma_1, \dots, \gamma_k; \rho)| \leq \exx_{x \in G} f(\gamma_1(x)) \dots f(\gamma_k(x)) \leq |G|^{-1} |B(\gamma_1, \dots, \gamma_k; \rho + \eta)|.\]

By assumption on the sizes of Bohr sets, we obtain
\[\Big||B(\gamma_1, \dots, \gamma_k; \rho)| - |G|\exx_{x \in G} f(\gamma_1(x)) \dots f(\gamma_k(x))\Big| \leq \varepsilon |G|.\]

We use Fourier analysis on $\mathbb{T}$ to estimate the expectation $\ex_{x \in |G|} f(\gamma_1x) \dots f(\gamma_kx)$. To simplify the notation, we use integers directly instead of continuous characters on $\mathbb{T}$, but we still use small greek letters to denote frequencies. The key property of $f$ is that it can be well-approximated in the $L^\infty$ sense by a finite sum coming from the truncation of its inverse Fourier transform. An important and standard fact is that the large spectrum of an interval consists of small values.

\begin{claim}\label{fourierboundsintervalconv}Let $\xi \in \mathbb{Z} \setminus \{0\}$. Then $|\hat{f}(\xi)| \leq \eta^{-1}/\xi^2$.\end{claim}
\begin{proof}[Proof of the claim]The Fourier coefficient at $\xi$ of an arbitrary interval $[-t,t]$ is given by
\[\int_{-t}^t \on{e}(-\xi x) dx = \frac{\on{e}(\xi t) - \on{e}(-\xi t)}{2 \pi i \xi},\]
which is at most $\frac{1}{\pi |\xi|}$ in absolute value. Since $\hat{f}(\xi) = \eta^{-1}\hat{I}(\xi) \hat{J}(\xi)$, the claim follows.\end{proof}

This rapid decay of Fourier coefficients allows us to deduce the desired approximation property for $f$.

\begin{claim}\label{intervalFCdecay}Let $S \subseteq \mathbb{Z}$ be a finite set that contains the interval $[-L, L]$ for some $L > 0$ (and possibly some other elements). Then
\[\Big|f(x) - \sum_{\xi \in S} \hat{f}(\xi) \on{e}(\xi x)\Big| \leq 2\eta^{-1}/L\]
for all $x \in \mathbb{T}$.\end{claim}

\begin{proof}[Proof of the claim]Using the inverse Fourier transform (which converges to $f(x)$ by Dirichlet's theorem, as $f$ is piecewise smooth) and Claim~\ref{fourierboundsintervalconv}, we obtain
\begin{align*}\Big|f(x) - \sum_{\xi \in S} \hat{f}(\xi) \on{e}(\xi x)\Big| = \,&\Big|\sum_{\xi \notin S} \hat{f}(\xi) \on{e}(\xi x)\Big| \leq \sum_{\gamma \notin S} |\hat{f}(\xi)| \leq\, 2\eta^{-1}\sum_{\xi = L + 1}^\infty \frac{1}{\xi^2} \leq \frac{2\eta^{-1}}{L}.\qedhere\end{align*}
\end{proof}

Let $K = 2 ek \eta^{-1} \varepsilon^{-1}$. Write $s(x) = \sum_{\xi \in [-K, K]} \hat{f}(\xi) \on{e}(\xi x)$. Thus, $\|f-s\|_{L^\infty} \leq \frac{\varepsilon}{ ek}$. In particular, $\|s\|_{L^\infty} \leq 1 + \frac{1}{k}$. Using this, we obtain
\begin{align}&\Big|\exx_{x \in G} f(\gamma_1(x)) \dots f(\gamma_k(x)) - \exx_{x \in G} s(\gamma_1(x)) \dots s(\gamma_k(x))\Big|\nonumber\\
&\hspace{2cm} \leq \sum_{i \in [k]} \Big|\exx_{x \in G} f(\gamma_1(x)) \dots f(\gamma_{i-1}(x))\, (f(\gamma_i(x)) - s(\gamma_i(x)))\, s(\gamma_{i+1}(x)) \dots s(\gamma_k (x))\Big|\nonumber\\
&\hspace{2cm} \leq k \Big(1 + \frac{1}{k}\Big)^k \|f - s\|_{L^\infty} \leq \varepsilon.\label{fromftofcineq}\end{align}

By our work so far, we have
\[\bigg||G|^{-1}|B(\gamma_1, \dots, \gamma_k; \rho)| - \exx_{x \in G} \Big(\sum_{a_1 \in [-K, K]} \hat{f}(a_1) \on{e}(a_1 \gamma_1(x))\Big) \cdots \Big(\sum_{a_k \in [-K, K]} \hat{f}(a_k) \on{e}(a_k\gamma_k(x))\Big)\bigg| \leq 2\varepsilon.\]

Finally,
\begin{align*} &\exx_{x \in G} \Big(\sum_{a_1 \in [-K, K]} \hat{f}(a_1) \on{e}\Big(a_1\gamma_1(x)\Big) \Big) \cdots \Big(\sum_{a_k \in [-K, K]} \hat{f}(a_k) \on{e}\Big(a_k \gamma_k (x)\Big) \Big) \\
&\hspace{2cm} = \sum_{a_1, \dots, a_k \in [-K, K]} \hat{f}(a_1)  \cdots \hat{f}(a_k) \exx_{x \in G} \on{e}\Big((a_1\gamma_1 + \dots + a_k\gamma_k)(x)\Big)\\
&\hspace{2cm} = \sum_{a_1, \dots, a_k \in [-K, K]} \id(a_1\gamma_1 + \dots + a_k \gamma_k = 0) \hat{f}(a_1)  \cdots \hat{f}(a_k).\end{align*}
The Fourier coefficients $\hat{f}(a)$ are the desired quantities $c_a$, thus
\begin{align*}c_a = \int_0^1 f(x) \on{e}(-a x) dx = &\eta^{-1}\int_{y = -\rho - \eta/2}^{\rho + \eta/2}\int_{z = -\eta/2}^{\eta/2} \on{e}(-a (y + z)) dy dz\\
 = &\eta^{-1}\frac{\Big(\on{e}(a (\rho + \eta/2)) - \on{e}(-a (\rho + \eta/2))\Big)\Big(\on{e}(a \eta/2) - \on{e}(-a \eta/2)\Big)}{4 \pi^2 a^2}.\end{align*}
Since $\|f\|_{L^\infty} \leq 1$ we also have $c_a \in \mathbb{D}$.\end{proof}

For a sequence $s_1, \dots, s_k$ of elements of an abelian group, we define its \emph{annihilator lattice} $\Lambda_\perp(s_1, \dots,$ $s_k)\subset \mathbb{Z}^k$ as $\Lambda_\perp(s_1, \dots, s_k) = \{(\lambda_1, \dots, \lambda_k) \in \mathbb{Z}^k \colon \sum_{i \in [k]} \lambda_i s_i = 0\}$. Previous proposition tells us that the size of a Bohr set $B(\gamma_1, \dots, \gamma_k; \rho)$ is essentially determined by the intersection of the annihilator lattice $\Lambda_\perp(\gamma_1, \dots, \gamma_k)$ with a bounded box in $\mathbb{Z}^k$.\\

We may reuse most of the proof above to prove that the large Fourier coefficients of a weakly-regular Bohr set come from linear combinations of the defining characters.

\begin{proposition}[Large Spectrum of Bohr sets]\label{bohrsizeLargeFC} Let $k \in \mathbb{N}$ and let $\rho, \eta, \varepsilon > 0$. Then, there exist a positive integer $K \leq O(k \varepsilon^{-1} \eta^{-1})$ for which the following holds.\\
\indent Suppose that $\gamma_1, \dots, \gamma_k \in \hat{G}$ and that 
\[|B(\gamma_1, \dots, \gamma_k; \rho + \eta) \setminus B(\gamma_1, \dots, \gamma_k; \rho)| \leq \varepsilon |G|/2.\]
Let $\chi \in \hat{G}$ be such that $\Big|\fco B(\gamma_1, \dots, \gamma_k; \rho) \fcc(\chi)\Big| \geq \varepsilon$. Then $\chi = a_1 \gamma_1 + \dots + a_k \gamma_k$ for some integers $a_1, \dots, a_k$ such that $|a_i| \leq K$.\end{proposition}

\begin{proof}Let $f$ be as in the previous proof. Then
\[\Big|\fco B(\gamma_1, \dots, \gamma_k; \rho) \fcc(\chi) - \exx_{x \in G} f(\gamma_1(x)) \dots f(\gamma_k(x)) \on{e}(- \chi x)\Big| \leq \varepsilon/2\]
so we have
\[\Big|\exx_{x \in G} f(\gamma_1 (x)) \dots f(\gamma_k (x)) \on{e}(- \chi x)\Big| \geq \varepsilon/2.\]
Using Claims~\ref{fourierboundsintervalconv} and~\ref{intervalFCdecay} as in the previous proof in inequality~\eqref{fromftofcineq} we see that for $K = 8ek\eta^{-1}\varepsilon^{-1}$ (the only difference is that $K$ is slightly larger this time) one has 
\[\Big|\exx_{x \in G} \Big(\sum_{a_1 \in [-K, K]} \hat{f}(a_1) \on{e}(a_1 \gamma_1(x))\Big) \dots \Big(\sum_{a_k \in [-K, K]} \hat{f}(a_k) \on{e}(a_k \gamma_k(x))\Big)  \on{e}(- \chi(x))\Big| \geq \varepsilon / 4.\]
In particular
\begin{align*}0 \not= &\sum_{a_1 \in [-K, K], \dots, a_k  \in [-K, K]} \hat{f}(a_1)\dots \hat{f}(a_k) \exx_{x \in G} \on{e}\Big((a_1 \gamma_1 + \dots a_k \gamma_k - \chi)(x)\Big)\\
&\hspace{2cm}= \sum_{a_1 \in [-K, K], \dots, a_k  \in [-K, K]} \hat{f}(a_1)\dots \hat{f}(a_k) \id\Big(a_1 \gamma_1 + \dots a_k \gamma_k = \chi\Big).\end{align*}
Thus there are $a_1, \dots, a_k \in [-K, K]$ such that $a_1 \gamma_1 + \dots a_k \gamma_k = \chi$.\end{proof}

Our next result about Bohr sets is a generalization of the fact that the dual of the sum of two subspaces of a vector space is given by the intersection of their duals. For a set $\Gamma = \{\gamma_1, \dots, \gamma_k\}$ of elements of an abelian group and a nonnegative integer $R$, we write $\langle \Gamma \rangle_R$ for the set of all linear combinations $\lambda_1\gamma_1 + \dots + \lambda_k \gamma_k$, where $\lambda_1, \dots, \lambda_k \in [-R, R] \subseteq \mathbb{Z}$. 

\begin{theorem}\label{bohrSum}Let $\Gamma_1, \Gamma_2 \subset \hat{G}$ and $\rho_1, \rho_2 \in (0,1)$. Then there is a positive integer $R \leq (2\rho_1^{-1})^{O(|\Gamma_1|)} + (2\rho_2^{-1})^{O(|\Gamma_2|)}$ such that
\[B(\langle \Gamma_1 \rangle_R \cap \langle \Gamma_2 \rangle_R, 1/4) \subseteq B(\Gamma_1, \rho_1) + B(\Gamma_2, \rho_2).\]
\end{theorem}

\begin{proof}The proof is based on the Bogolyubov argument. Let $\sigma_1 \in [\rho_1/8, \rho_1/4]$ and $\sigma_2 \in [\rho_2/8, \rho_2/4]$, which we shall specify later. We need some flexibility in the choice of radii in order to satisfy the weak-regularity hypothesis for the resulting Bohr sets. Write $B_1 = B(\Gamma_1, \sigma_1)$ and $B_2 = B(\Gamma_2, \sigma_2)$. Let $f = B_1 \ast B_2$. Hence, $f(x) \in [0,1]$ for all $x$ and $\delta = \ex_x f(x)$ satisfies $\delta \geq \Big(\frac{\rho_1}{8}\Big)^{|\Gamma_1|}\Big(\frac{\rho_2}{8}\Big)^{|\Gamma_2|}$. Observe that 
\[\on{supp} f \ast f \ast f \ast f \subseteq  B(\Gamma_1, \rho_1) + B(\Gamma_2, \rho_2).\]
Using the inverse Fourier transform, we obtain
\[f \ast f \ast f \ast f(x) = \sum_{\gamma \in \hat{G}} \hat{f}(\gamma)^4 \on{e}(\gamma(x)).\]
Let $\Gamma$ be the set of all characters $\gamma \in \hat{G}$ such that $|\hat{f}(\gamma)| \geq \delta^2/10$. Then
\[\Big|f \ast f \ast f \ast f(x) - \sum_{\gamma \in \Gamma} \hat{f}(\gamma)^4 \on{e}(\gamma(x))\Big| \leq \sum_{\gamma \in \hat{G} \setminus \Gamma} \Big|\hat{f}(\gamma)^4 \on{e}(\gamma(x))\Big| \leq \max_{\gamma \notin \Gamma} |\hat{f}(\gamma)|^2 \cdot \sum_{\gamma \in \hat{G}} |\hat{f}(\gamma)|^2 \leq \frac{\delta^4}{100} \exx_{x \in G} |f(x)|^2 \leq \frac{\delta^4}{100}.\]

Note that $\hat{f}(\gamma) \in \mathbb{R}$ for all $\gamma \in \hat{G}$ since $f$ is a symmetric function. Furthermore, we have $\on{Re} \on{e}(\gamma(x)) \geq 0$ for $\gamma \in \Gamma$ and $x \in B(\Gamma; 1/4)$, and $\hat{f}(0)^4 = \delta^4$. Therefore, $x \in B(\Gamma; 1/4)$ gives $f \ast f \ast f \ast f(x) \not= 0$, which further implies that $x \in B(\Gamma_1, \rho_1) + B(\Gamma_2, \rho_2)$. \\

\indent Finally, since $\hat{f} = \hat{B}_1 \hat{B}_2$, we have that every $\gamma \in \Gamma$ satisfies $|\hat{B}_1(\gamma)|, |\hat{B}_2(\gamma)| \geq \delta^2/ 10$. We now pick $\sigma_1$ and $\sigma_2$. Start with $\sigma_1 = \rho_1/4$ and move towards $\rho_1/8$ in steps of $\rho_1 \delta^2 / 200$ each time we have
\[\Big|B(\Gamma_1, \sigma_1) \setminus B(\Gamma_1; \sigma_1 - \rho_1 \delta^2/200)\Big| \geq \frac{\delta^2}{20}|G|.\]
At the end of the procedure we obtain $\sigma_1 \geq \rho_1 / 8$ and the required weak regularity of $B_1$. The same arguments applies to $\sigma_2$. Proposition~\ref{bohrsizeLargeFC} implies that $\gamma \in \langle \Gamma_1 \rangle_R \cap \langle \Gamma_2 \rangle_R$ for some $R \leq O((|\Gamma_1| \rho^{-1}_1 + |\Gamma_2|\rho^{-1}_2) \delta^{-4}) \leq (2\rho_1^{-1})^{O(|\Gamma_1|)} + (2\rho_2^{-1})^{O(|\Gamma_2|)}$, completing the proof.\end{proof}

\noindent\textbf{Bohr sets inside coset progressions.} A fundamental fact about Bohr sets is that they contain dense coset progressions of bounded rank. We now prove the opposite statement, namely that dense bounded rank coset progressions contain Bohr sets of relatively large radius and relatively small frequency set.

\begin{proposition}\label{cosettobohrset}Let $C$ be a symmetric coset progression of rank $r$ and density $\alpha$ inside $G$. Then there is a positive integer $R \leq (r \log (\alpha^{-1}))^{O(1)}$ and characters $\chi_1, \dots, \chi_R \in \hat{G}$ such that $B(\chi_1, \dots, \chi_R; 1/(4R)) \subseteq C$.\end{proposition}

\begin{proof}Let $C = [-N_1, N_1] \cdot x_1 + \dots + [-N_r, N_r] \cdot x_r + H$ be the canonical form of $C$. Let $d$ be a positive integer to be chosen later. Ignore those $N_i$ that satisfy $N_i \leq d^3$ by setting them to 0 instead. This allows us to assume $N_i \geq d^3$ without loss of generality, at the cost of decreasing the density of $C$ to $\alpha d^{-3r}$. Consider another symmetric coset progression of rank $r$ given by shrinking $C$ by a factor of roughly $d$, namely $Q = [-N'_1, N'_1] \cdot x_1 + \dots + [-N'_r, N'_r] \cdot x_r$, where $N'_i = \lfloor \frac{N_i}{d}\rfloor$. Since $N_i \geq d^3$ for all $i \in [r]$, we have $d Q \subseteq C \subseteq (d+1)Q$.\\

Using the inverse Fourier transform, we have 
\[\underbrace{Q \ast Q \ast \dots \ast Q}_{d}(x) = \sum_{\gamma \in \hat{G}} \hat{Q}(\gamma)^d \on{e}(\gamma(x)).\]

Let $\delta$ be the density of $Q$. Thus $\delta \geq d^{-4r} \alpha$. Let $\Gamma$ be the set of $\gamma \in \hat{G}$ such that $|\hat{Q}(\gamma)| \geq \xi$, for some $\xi$ to be chosen. Then

\[\Big|\underbrace{Q \ast Q \ast \dots \ast Q}_{d}(x) - \Big(\delta^d  + \sum_{\gamma \in \Gamma \setminus \{0\}} \hat{Q}(\gamma)^d \on{e}(\gamma(x))\Big)\Big| \leq \xi^{d-2} \delta.\]

Note that $Q$ is symmetric and thus $\hat{Q}(\gamma)^d \in \mathbb{R}_{\geq 0}$ when $d$ is an even integer. Hence, provided we take $\xi = \frac{1}{2} \delta^{1 + \frac{1}{d-2}}$, $B(\Gamma, 1/4) \subseteq P$. By Chang's theorem~\cite{ChangFreiman} (see also Lemma 4.36 in~\cite{TaoVuBook}), there exists a set $\Psi \subseteq \hat{G}$ of size
\[R \leq O(\delta^{-2/(d-2)} \log (\delta^{-1}))\]
such that every element of $\Gamma$ is a $\{-1,0,1\}$-linear combination of elements in $\Psi$. In particular, 
\[B(\Psi; 1/(4R)) \subseteq B(\Gamma, 1/4) \subseteq P.\]
Recall that $\delta \geq \alpha d^{-4r}$. Thus,
\[R \leq O(\delta^{-2/(d-2)} \log (\delta^{-1})) \leq O\Big(\alpha^{-2/(d-2)} d^{8r/(d-2)} (\log \alpha^{-1} + 4r \log d)\Big).\]
We may pick an even integer $d \leq (r\log \alpha^{-1})^{O(1)}$ to finish the proof.\end{proof}

We may use this proposition to get a conclusion of Theorem~\ref{bogRuzsa1} which involves Bohr sets rather than coset progressions.

\begin{corollary}\label{bogRuzsa2}Suppose that $A \subset G$ is a set of density $\delta$. Then $2A - 2A$ contains a Bohr set of codimension at most $\log^{O(1)}\delta^{-1}$ and radius at least $\Big(\log^{O(1)}\delta^{-1}\Big)^{-1}$.\end{corollary}

\section{Quantitative fundamental theorem of lattices}\label{secLattice}

Recall that the fundamental theorem of lattices (Lemma 3.4 in~\cite{TaoVuBook}) states that lattices of rank $k$ (meaning the maximal linearly independent subset is of size $k$) are precisely of the form $\mathbb{Z} \cdot v_1 + \dots + \mathbb{Z} \cdot v_k$ for some independent $v_1, \dots, v_k$. Our goal is to prove a quantitative variant of this result. We need a preliminary lemma that says that nested sequences of lattices inside a fixed box terminate quickly. 

\begin{lemma}\label{nestedLattices}Let $k$ and $K$ be positive integers. Suppose that $\Lambda_1 \subseteq \Lambda_2 \subseteq \dots \subseteq \Lambda_r \subseteq \mathbb{Z}^k$ are lattices such that for each $i \in [r-1]$ the set-difference $\Lambda_{i + 1} \setminus \Lambda_i$ contains an element of the box $[-K, K]^k$. Then $r \leq O(k^2 (\log k + \log K))$.\end{lemma}

This lemma will be used later in the paper as well, so we opted to state it outside the proof of the main result on lattices.

\begin{proof}Without loss of generality we may assume that the given lattices are generated by elements in the box $[-K, K]^k$. Note that the ranks of lattices $\Lambda_1, \dots$ are non-decreasing and at most $k$, so it suffices to show that $m \leq O(k (\log k + \log K))$ when $\Lambda_1, \dots, \Lambda_m$ all have the same rank. By passing to a suitable subset of coordinates, we may in fact assume that the ranks are $k$. The fact that $\Lambda_1$ is generated by $k$ independent vectors with coordinates bounded by $K$ in absolute value implies the covolume bound $\det \Lambda_1 \leq k! K^k$. We have $\det \Lambda_{i+1} \leq \frac{1}{2} \det \Lambda_i$, so $m \leq O(k(\log k + \log K))$.\end{proof}

The following theorem is a quantitative version of the fundamental theorem of lattices. Instead of the full group $\mathbb{Z}^m$, we consider a truncated version, namely the set $\langle a_1, \dots, a_k \rangle_R$. We claim that every subset $B \subseteq \langle a_1, \dots, a_k \rangle_R$ is covered by bounded linear combinations of a bounded number of elements $b_1, \dots, b_\ell$ in $B$. In comparison, the fundamental theorem of lattices tells us that $B$ is contained in $\langle b_1, \dots, b_\ell \rangle$ for some $b_1, \dots, b_\ell \in \langle B \rangle$, where $\ell$ is exactly determined (it is the size of a maximal independent set). We need more flexibility in the choice of $\ell$ in the quantitative version as the example $B = \{2,3\} \subseteq \mathbb{Z}$ shows. The main difference is that in our case we need elements from the original set $B$ rather than its span $\langle B \rangle$.

\begin{theorem}\label{quantLattice}Let $G$ be a finite abelian group, let $a_1, \dots, a_k \in G$ and let $R$ be a positive integer. Suppose that $B \subset \langle a_1, \dots, a_k \rangle_R$ is a non-empty set. Then, there exist positive integers $\ell \leq O(k^2 (\log k + \log R))$, $S \leq O((2Rk)^{k+ 3})$ and elements $b_1, \dots, b_\ell \in B$ such that $B \subseteq \langle b_1, \dots, b_\ell\rangle_{S}$.\end{theorem}

\begin{proof}We pass to $\mathbb{Z}^k$ by choosing an arbitrary $k$-tuple $(x_1, \dots, x_k) \in \mathbb{Z}^k$ such that $b = x_1 a_1 + \dots + x_k a_k$ and $\|x\|_{L^\infty} \leq R$ for each $b \in B$. Let $X \subseteq \mathbb{Z}^k$ be the set of the chosen $k$-tuples. We iteratively find elements $x^1, x^2, \dots \in X$ by adding arbitrary $x^{i+1} \in X \setminus \langle x^1, \dots, x^i\rangle_{\mathbb{Z}}$ to the list every time $X \not\subseteq \langle x^1, \dots, x^i\rangle_{\mathbb{Z}}$. By Lemma~\ref{nestedLattices} this procedure terminates after $m \leq O(k^2 (\log k + \log R))$ steps. It turns out that elements of $X$ are necessarily $\mathbb{Z}$-linear combinations of $x^1, x^2, \dots$ with all coefficients bounded. This is the content of the next lemma.

\begin{lemma}\label{latticeAdjugateBOund}Suppose that $w \in \langle z^1, \dots, z^r \rangle_{\mathbb{Z}}$ for some $z^1, \dots, z^r \in \mathbb{Z}^k$ with $\|z^i\|_{L^\infty} \leq K_1$ for $i \in [r]$ and $\|w\|_{L^\infty} \leq K_2$. Then there are $\lambda_1, \dots, \lambda_r \in \mathbb{Z}$ such that $|\lambda_i| \leq k! K_1^{k + 1}(K_2 + r)$ for $i \in [r]$ and $w = \sum_{i \in [r]} \lambda_i z^i$.\end{lemma}

\begin{proof}[Proof of the lemma.] Without loss of generality, $z^1, \dots, z^m$ form a maximal $\mathbb{Q}$-linearly independent subset of $z^1, \dots, z^r$. Reordering the coordinates if necessary, we may assume that $z^1, \dots, z^m$ are still independent when restricted to the first $m$ coordinates. Thus, we may assume without loss of generality that $m = k$.\\
\indent Consider the $m \times m$ matrix $Z = (z^1\,\,z^2\,\,\dots\,\,z^m)$, which is integral and invertible. For any $t \in \mathbb{Z}^m$ we have unique $\mu_1, \dots, \mu_m \in \mathbb{Q}$ such that $t = \sum_{i \in [m]} \mu_i z^i$. In the matrix notation, $\mu \in \mathbb{Q}^m$ is the unique solution of the equation $t = Z \mu$. Using the adjugate matrix we obtain
\[\det Z \,\mu = \on{adj} Z \,t.\]
In particular, $\det Z \,\mu$ is an integral vector. Using this observation for $z^{m+1}, \dots, z^r$ in place of $t$, it follows that $\det Z z^j$ is a $\mathbb{Z}$-linear combination of $z^1, \dots, z^m$ for every $j \in [m + 1, r]$. In particular, this means that we may find $\lambda_1, \dots, \lambda_r \in \mathbb{Z}$ such that $|\lambda_{m+1}|, \dots, |\lambda_r| \leq |\det Z|$ and $w = \sum_{i \in [r]} \lambda_i z^i$. It remains to bound the coefficients $\lambda_1, \dots, \lambda_m$. Using the equality
\[w - \sum_{i \in [m + 1, r]} \lambda_i z^i =  Z (\lambda_1\,\,\dots\,\,\lambda_m)^T\]
and the invertibility of $Z$ we have 
\[(\lambda_1\,\,\dots\,\,\lambda_m)^T = \frac{1}{\det Z} \on{adj} Z \Big(w - \sum_{i \in [m + 1, r]} \lambda_i z^i\Big).\]
The bounds so far and the fact that $|\det Z| \geq 1$ imply that 
\[|\lambda_1|, \dots, |\lambda_m| \leq \frac{1}{|\det Z|} m! K_1^{m}\Big(K_2 + r |\det Z| K_1\Big) \leq m! K_1^{m + 1}(K_2 + r),\]
concluding the proof.\end{proof} 

Thus, $X \subset \langle x^1, \dots, x^m\rangle_S$ for some $S \leq k!R^{k+1}(R + m)$. Finally, we return from $\mathbb{Z}^k$ to the group $G$. Let $b_i = \sum_{j \in [k]} x^i_j a_j$ for $i \in [m]$. Let $b \in B$ be an arbitrary element. Then $b = \sum_{i \in [k]} \lambda_i a_i$ for some $\lambda \in X$. Thus, there are coefficients $\mu_1, \dots, \mu_m \in [-S, S]$ such that $\lambda = \sum_{i \in [m]} \mu_i x^i$. In particular
\[b = \sum_{i \in [k]} \lambda_i a_i = \sum_{i \in [k], j \in [m]} \mu_j x^j_i a_i = \sum_{j \in [m]} \mu_j b_j,\]
proving that $B \subseteq \langle b_1, \dots, b_m\rangle_S$.\end{proof}

\vspace{\baselineskip}

\section{Quasirandomness of Bilinear Bohr Varieties}\label{secQRBohr}

\hspace{18pt}Let $C$ be a symmetric proper coset progression of rank $d$ inside $H$, let $\Gamma \subseteq \hat{G}$ be a set of characters and let $L_1, \dots, L_r \colon C \to \hat{G}$ be Freiman-linear maps on the coset progression $C$. Define the \emph{bilinear Bohr variety} 
\[W = (B(\Gamma; \rho) \times C) \cap \Big(\cup_{y \in C} B(L_1(y), \dots, L_r(y); \rho) \times \{y\}\Big).\]
\indent We view $W$ as a bipartite graph between the vertex classes $B(\Gamma; \rho)$ and $C$. In this section we study quasirandomness properties of bilinear Bohr varieties, generalizing the results in~\cite{U4paper} and~\cite{genPaper}. The main result is an algebraic regularity lemma which essentially enables us to partition $W$ into quasirandom pieces of the form $W \cap (G \times C_i)$, where $C = C_1 \cup \dots \cup C_m$ is a partition of the vertex class $C$ into further coset progressions. There is a minor caveat that we might need to decrease the radius $\rho$ slightly on each piece.\\
\indent The particular choice of $B(\Gamma; \rho) \times C$ instead of a product of two Bohr sets or two coset progressions is of minor importance, given that Bohr sets and coset progressions are essentially equivalent, but it seems to give the cleanest statement of the algebraic regularity lemma. 

\begin{theorem}\label{algreglemma}Let $G$ and $H$ be finite abelian groups. Let $C$ be a symmetric proper coset progression of rank $d$ inside the group $H$, let $\Gamma \subseteq \hat{G}$ and let $L_1, \dots, L_r \colon C \to \hat{G}$ be Freiman-linear maps. Let $\rho > 0$ and $\eta > 0$ be given. We may partition $C$ into further proper coset progressions $C_1, \dots, C_m$ of rank at most $d$, where
\[m \leq \exp\Big(d^{O(1)} r^{O(1)} |\Gamma|^{O(1)}\log^{O(1)} \eta^{-1} \log^{O(1)} \rho^{-1}\Big),\]
such that for each $i \in [m]$ there exist reals $\delta_i > 0$ and $\rho_i \in [\rho/2, \rho]$ such that the following two quasirandomness properties hold.
\begin{itemize}
\item[\textbf{(i)}] For at least a $1 - \eta$ proportion of all elements $y \in C_i$ we have 
\[\Big||B(\Gamma \cup \{L_1(y), \dots, L_r(y)\}; \rho_i)| - \delta_i |B(\Gamma; \rho_i)|\Big| \leq \eta |G|.\]
\item[\textbf{(ii)}] For at least a $1 - \eta$ proportion of all pairs $(y, y') \in C_i \times C_i$ we have
\[\Big||B(\Gamma \cup \{L_1(y), \dots, L_r(y), L_1(y'), \dots, L_r(y')\}; \rho_i)| - \delta^2_i |B(\Gamma; \rho_i)|\Big| \leq \eta |G|.\]
\end{itemize}
\end{theorem}

Properties \textbf{(i)} and \textbf{(ii)} are sufficient to guarantee that the corresponding bipartite graph is quasirandom, see Appendix~\ref{qrAppendix} for a brief discussion of quasirandom bipartite graphs.

\begin{proof} Let $C = [-N_1, N_1] \cdot v_1 + \dots + [-N_d, N_d] \cdot v_d + H_0$ be a canonical form of $C$. We iteratively find further proper coset progressions $C = Q_0 \supseteq Q_1 \supseteq Q_2 \supseteq \dots$ such that each $Q_i$ is of the form
\[Q_i = [-N^{(i)}_1, N^{(i)}_1] \cdot \ell^{(i)}_1 v_1 + \dots + [-N^{(i)}_d, N^{(i)}_d] \cdot \ell^{(i)}_d v_d + H_i\]
for suitable integers $N^{(i)}_j, \ell^{(i)}_j$ and subgroups $H_j$. The bounds
\[\frac{2N_1 + 1}{2N^{(i)}_1 + 1}, \dots, \frac{2N_d + 1}{2N^{(i)}_d + 1}, \frac{|H_0|}{|H_i|} \leq \Big(\eta^{-2} 100^{d+2} (2K + 1)^{4r + 2|\Gamma|}\Big)^{i}\]
will hold at each step $i$, where $K$ is a parameter that will be defined later, which will satisfy $K = O((|\Gamma| + r) \eta^{-3}\rho^{-1})$ and will not depend on $i$. Along the way, we shall also construct an increasing sequence of lattices $\{0\} = \Lambda_0 \subseteq \Lambda_1 \subseteq \Lambda_2 \subseteq \dots $ in $\mathbb{Z}^r$ such that for each $y \in Q_i$ and $\lambda \in \Lambda_i$ we have $\lambda_1 L_1(y) + \dots + \lambda_r L_r(y) = 0$.\\
\indent Let $M \subset \mathbb{Z}^\Gamma$ be the annihilator lattice for the set $\Gamma$, which is the lattice of those tuples $(\nu_\gamma)_{\gamma \in \Gamma}$ that satisfy $\sum_{\gamma \in \Gamma} \nu_\gamma \gamma = 0$.\\

Suppose that we have completed the $s$\tss{th} step and thus constructed the proper coset progression $Q_s$ and the lattice $\Lambda_s$. Let 
\[Q'_s = [-N^{(s)}_1/2, N^{(s)}_1/2] \cdot \ell^{(s)}_1 v_1 + \dots + [-N^{(s)}_d/2, N^{(s)}_d/2] \cdot \ell^{(s)}_d v_d + H_s\]
which is a slightly shrunk version of $Q_s$, chosen so that $Q'_s - Q'_s \subseteq Q_s$. Let us partition the initial coset progression $C$ into proper coset progressions $S_1, \dots, S_m$ by intersecting it with translates of $Q'_s$. It is possible to achieve that $|S_i| \geq 100^{-d}|Q_s|$ for all $i \in [m]$. For each $i \in [m]$, we first find a suitable $\rho_i \in [\rho/2, \rho]$. Consider $\rho - \frac{j \eta^2 \rho}{1000}$ for $j \in [500 \eta^{-2}]$ as candidates for $\rho_i$. There is such a value of $\rho_i$ such that for at least a $1 - \frac{\eta}{10}$ proportion of $y \in S_i$ we have
\begin{align}\Big|B\Big(\Gamma \cup \{L_1(y), \dots, L_r(y)\}; \rho_i + \frac{\eta^2 \rho}{1000}\Big) \setminus B\Big(\Gamma \cup \{L_1(y), \dots, L_r(y)\}; \rho_i\Big)\Big| \leq \frac{\eta}{10} |G|,\label{wreg1piece}\end{align}
for at least a $1-\frac{\eta}{10}$ proportion of the pairs $(y,y') \in S_i \times S_i$ we have 
\begin{align}&\Big|B\Big(\Gamma \cup \{L_1(y), \dots, L_r(y), L_1(y'), \dots, L_r(y')\}; \rho_i + \frac{i \eta^2 \rho}{1000}\Big)\nonumber\\
&\hspace{2cm}\setminus B\Big(\Gamma \cup \{L_1(y), \dots, L_r(y), L_1(y'), \dots, L_r(y')\}; \rho_i\Big)\Big| \leq \frac{\eta}{10} |G|\label{wreg2piece}\end{align}  
and 
\begin{equation}\Big|B\Big(\Gamma; \rho_i + \frac{\eta^2 \rho}{1000}\Big) \setminus B(\Gamma; \rho_i)\Big| \leq \frac{\eta}{10} |G|.\label{wreg3piece}\end{equation}

Let us fix some $S_i$. The next claim shows that $S_i$ gives rise to a quasirandom piece unless we obtain new vanishing linear combinations of characters given by values of maps $L_1, \dots, L_r$. Let $K = O((|\Gamma| + r) \eta^{-3}\rho^{-1})$ be the quantity and $c_i \in \mathbb{D}$ for $i \in [-K, K]$ the constants provided by Proposition~\ref{bohrsizeformula} such that whenever $\gamma_1, \dots, \gamma_\ell \in \hat{G}$ are characters for $\ell \leq 2r + |\Gamma|$ with the weak regularity property
\begin{equation}\Big||B(\gamma_1, \dots, \gamma_\ell; \rho_i + \eta^2 \rho/1000)| - |B(\gamma_1, \dots, \gamma_\ell; \rho_i)|\Big| \leq \frac{\eta}{10} |G|\label{wregconditionQR}\end{equation}
then 
\begin{equation}\label{QRbohrconclusion}\Big||B(\gamma_1, \dots, \gamma_\ell; \rho_i)| - \sum_{a_1, \dots, a_\ell \in [-K, K]} \id(a_1 \gamma_1 + \dots + a_\ell \gamma_\ell = 0) c_{a_1} \dots c_{a_\ell} |G|\Big| \leq \frac{\eta}{5}|G|.\end{equation}

\begin{claim}\label{qrsuffconds}Suppose that at least a $1 - \frac{\eta}{10}$ proportion of all elements $y \in S_i$ have the property that if the equality
\begin{equation}\sum_{\gamma \in \Gamma} \nu_\gamma \gamma + \lambda_1 L_1(y) + \dots + \lambda_r L_r(y) = 0\label{latticeclaimproperty}\end{equation}
holds for some $\nu_\gamma, \lambda_j \in [-K, K]$ then $\nu \in M$ and $\lambda \in \Lambda_s$. Suppose also that at least $1 - \frac{\eta}{10}$ proportion of all pairs $(y, y') \in S_i \times S_i$ have the property that if the equality
\[\sum_{\gamma \in \Gamma} \nu_\gamma \gamma + \lambda_1 L_1(y) + \dots + \lambda_r L_r(y) + \lambda'_1 L_1(y') + \dots + \lambda'_r L_r(y') = 0\]
holds for some $\nu_\gamma, \lambda_j, \lambda'_j \in [-K, K]$ then $\nu \in M$ and $\lambda, \lambda' \in \Lambda_s$. Then the properties \textbf{(i)} and \textbf{(ii)} hold.
\end{claim}

\begin{proof}Define $\delta_i$ to be the quantity
\[\sum_{\lambda \in \Lambda_s \cap [-K, K]^r} c_{\lambda_1} \dots c_{\lambda_r}\]
and define $\delta_0$ as
\[\sum_{\nu \in M \cap [-K, K]^\Gamma} \prod_{\gamma \in \Gamma}c_{\nu_\gamma}.\]
From inequalities~\eqref{wreg3piece} and~\eqref{QRbohrconclusion} we conclude that
\[\Big||B(\Gamma; \rho_i)| - \delta_0 |G|\Big| \leq \frac{\eta}{5}|G|.\]
Note that at least a $1-\eta/5$ proportion of all $y \in S_i$ obey~\eqref{wreg1piece} and the property involving~\eqref{latticeclaimproperty} described in the claim. Then the Bohr set $B(\Gamma \cup \{L_1(y), \dots, L_r(y)\}; \rho_i)$ satisfies the condition~\eqref{wregconditionQR} so the inequality~\eqref{QRbohrconclusion} gives
\begin{align*}\Big||B(\Gamma \cup \{L_1(y), \dots, L_r(y)\}; \rho_i)|\, - \sum_{\ssk{a \in [-K, K]^r\\\nu \in [-K, K]^\Gamma}} &\id\Big(\sum_{\gamma \in \Gamma} \nu_\gamma \gamma + a_1 L_1(y) + \dots + a_r L_r(y) = 0\Big)\\
&\hspace{4cm} \Big(\prod_{\gamma \in \Gamma} c_{\nu_\gamma}\Big)c_{a_1} \dots c_{a_r} |G|\Big| \leq \frac{\eta}{5}|G|,\end{align*}
while we have the equality
\[\sum_{\ssk{a \in [-K, K]^r\\\nu \in [-K, K]^\Gamma}} \id\Big(\sum_{\gamma \in \Gamma} \nu_\gamma \gamma + a_1 L_1(y) + \dots + a_r L_r(y) = 0\Big) \Big(\prod_{\gamma \in \Gamma} c_{\nu_\gamma}\Big)c_{a_1} \dots c_{a_r} = \delta_0 \delta_i.\]
Thus,
\[\Big||B(\Gamma \cup \{L_1(y), \dots, L_r(y)\}; \rho_i)| - \delta_i|B(\Gamma; \rho_i)|\Big| \leq 2\frac{\eta}{5}|G|,\]
proving the first part of the claim. The second property follows similarly.\end{proof}

Suppose that $S_i$ is not quasirandom in the sense that the properties \textbf{(i)} and \textbf{(ii)} do not hold simultaneously for the radius $\rho_i$. Then at least one of the two assumptions in Claim~\ref{qrsuffconds} fail, but in either case, we conclude that there are at least $\frac{\eta}{10}|S_i|^2$ pairs $(y, y') \in S_i \times S_i$ for which we have an equality
\[\sum_{\gamma \in \Gamma} \nu_\gamma \gamma + \lambda_1 L_1(y) + \dots + \lambda_r L_r(y) + \lambda'_1 L_1(y') + \dots + \lambda'_r L_r(y') = 0\]
where at least one of $\nu \in M, \lambda \in \Lambda_s, \lambda' \in \Lambda_s$ fails and all coefficients belong to $[-K, K]$. By averaging, we obtain such a linear combination that holds for at least $\frac{\eta}{10 (2K + 1)^{2r + |\Gamma|}}|S_i|^2$ pairs $(y,y')$ of elements in $S_i$ which we now fix.\\ 

Suppose first that $\lambda, \lambda' \in \Lambda_s$. Then $\lambda_1 L_1(y) + \dots + \lambda_r L_r(y) + \lambda'_1 L_1(y') + \dots + \lambda'_r L_r(y') = 0$ so $\sum_{\gamma \in \Gamma} \nu_\gamma \gamma = 0$, which means that $\nu \in M$ which is a contradiction. Therefore, without loss of generality, we may assume that $\lambda \notin \Lambda_s$. By the Cauchy-Schwarz inequality we may find at least $\frac{\eta^2}{100 (2K + 1)^{4r + 2|\Gamma|}}|S_i|^3$ triples $(y_1, y_2, y')$ such that 
\[\sum_{\gamma \in \Gamma} \nu_\gamma \gamma + \lambda_1 L_1(y_1) + \dots + \lambda_r L_r(y_1) + \lambda'_1 L_1(y') + \dots + \lambda'_r L_r(y') = 0\]
and 
\[\sum_{\gamma \in \Gamma} \nu_\gamma \gamma + \lambda_1 L_1(y_2) + \dots + \lambda_r L_r(y_2) + \lambda'_1 L_1(y') + \dots + \lambda'_r L_r(y') = 0.\]

Subtracting these equalities from one another and fixing a suitable $y'$, using the facts that $y_1, y_2, y_1 - y_2 \in C$ and the maps $L_i$ are Freiman-linear we obtain
\[\lambda_1 L_1(y_1 - y_2) + \dots + \lambda_r L_r(y_1 - y_2) = 0\]
for at least $\frac{\eta^2}{100 (2K + 1)^{4r + 2|\Gamma|}}|S_i|^2$ pairs $y_1, y_2 \in S_i$. By the way we defined $S_i$, we have $S_i - S_i \subseteq Q_s$ so we see that $\lambda_1 L_1(z) + \dots + \lambda_r L_r(z) = 0$ holds for at least
\[\frac{\eta^2}{100 (2K + 1)^{4r + 2|\Gamma|}}|S_i| \geq \frac{\eta^2}{100^{d+1} (2K + 1)^{4r + 2|\Gamma|}}|Q_s|\]
of $z \in Q_s$. But the map $z \mapsto \lambda_1 L_1(z) + \dots + \lambda_r L_r(z)$ is a Freiman-linear map on $Q_s$ so the subset consisting of $z$ such that $\lambda_1 L_1(z) + \dots + \lambda_r L_r(z) = 0$ is a Freiman-subgroup $F$ of $Q_s$. Apply Lemma~\ref{progsbgp} to find a further proper symmetric coset progression $Q_{s+1} \subseteq F$ with
\[\frac{2N^{(i)}_1 + 1}{2N^{(i + 1)}_1 + 1}, \dots, \frac{2N^{(i)}_d + 1}{2N^{(i+1)}_d + 1}, \frac{|H_i|}{|H_{i + 1}|} \leq \eta^{-2} 100^{d+2} (2K + 1)^{4r + 2|\Gamma|}\]
and let $\Lambda_{s+1} = \Lambda_s + \langle \lambda \rangle_{\mathbb{Z}}$.\\ 
\indent Since the set $(\Lambda_{s+1} \setminus \Lambda_s) \cap [-K, K]^r$ is non-empty for every $s$ (as it contains the $r$-tuple $\lambda$ chosen above), Lemma~\ref{nestedLattices} bounds the number of steps in the proof by $O(r^2(\log r + \log K))$.\end{proof}

\vspace{\baselineskip}

\section{Bilinear Bogolyubov argument}\label{secBilBA}

In this section we prove our main result which is stated slightly differently compared to the introduction. This time we use coset progressions in one variable and Bohr sets in the second.  

\begin{theorem}\label{bogruzsabilinear2}Let $G$ and $H$ be finite abelian groups and let $A \subseteq G \times H$ be a set of density $\delta$. Then there exist a proper symmetric coset progression $C \subseteq H$ of rank at most $\log^{O(1)} \delta^{-1}$ and size $|C| \geq \exp\Big(-\log^{O(1)} \delta^{-1}\Big) |H|$, a set $\Gamma \subseteq \hat{G}$ of size at most $\log^{O(1)} \delta^{-1}$, Freiman-linear maps $L_1, \dots, L_r \colon C \to \hat{G}$ for some positive integer $r \leq \log^{O(1)} \delta^{-1}$ and a positive quantity $\rho \geq \exp\Big(-\log^{O(1)} \delta^{-1}\Big)$ such that the bilinear Bohr variety
\[\Big\{(x,y) \in B(\Gamma; \rho) \times C \colon x \in B(L_1(y), \dots, L_r(y); \rho) \Big\}\]
is contained inside $\dhor\dver \dver \dhor \dver \dhor \dhor A$.\end{theorem}

Throughout the proof we use the following slice notation for subsets $A$ of $G \times H$. For $x \in G$, we write 
\[A_{x\bcdot} = \{y \in H \colon (x,y) \in A\}\]
which stands for the set of elements of $H$ lying in the column indexed by the element $x$, and analogously we write for a given $y \in H$
\[A_{\bcdot y} = \{x \in G \colon (x,y) \in A\}\]
for the set of elements of $G$ in the row indexed by $y$.

\begin{proof} The proof consists of 5 steps, where we take appropriate directional difference sets in each one (sometimes we take a double difference set which explains the discrepancy between the number of steps and the number of operator applications in the theorem). As we complete the steps of the proof we obtain an increasingly algebraically structured set in the iterated directional difference set.\\

\noindent\textbf{Step 1.} By averaging, there is a set $Y^1 \subseteq H$ of density $\delta/2$ such that $A_{\bcdot y}$ has size $|A_{\bcdot y}| \geq \frac{\delta}{2} |G|$ for each $y \in Y^1$. Perform a double horizontal convolution which results in the set $A^1 = \dhor \dhor A$. Applying Corollary~\ref{bogRuzsa2} for each $y \in Y^1$ we obtain a Bohr set $B(\Gamma_y, \rho)$ inside $A^1_{\bcdot y}$ of codimension at most $\log^{O(1)} \delta^{-1}$ and radius $\rho \geq \Big(\log^{O(1)}\delta^{-1}\Big)^{-1} \geq \exp\Big(-\log^{O(1)} \delta^{-1}\Big)$.\footnote{We may take $\rho$ to be independent of $y$ as Corollary~\ref{bogRuzsa2} gives a lower bound in terms of $\delta$ only and decreasing $\rho$ does not cause any harm.} Thus, the work in the first step of the proof results in the inclusion 
\[\bigcup_{y \in Y^1} B(\Gamma_y, \rho) \times \{y\}  \subseteq A^1.\]

\noindent\textbf{Step 2.} Let us now perform a single vertical convolution. We obtain $A^2 = \dver A^1$ which satisfies
\[A^2_{\bcdot y} = \bigcup_{\ssk{z \in Y^1\text{ s.t.}\\y + z \in Y^1}} \,\Big(B(\Gamma_{y + z}, \rho) \cap B(\Gamma_z, \rho)\Big) \times \{y\}\]
for all $y \in H$.\\

\noindent\textbf{Step 3.} In the third step, we make a horizontal convolution, thus obtaining $A^3 = \dhor A^2$. We see that
\[A^3_{\bcdot y} = \bigcup_{\ssk{z, w \in Y^1\text{ s.t.}\\y + z, y + w \in Y^1}} \,\Big(B(\Gamma_{y + z}, \rho) \cap B(\Gamma_z, \rho)\Big) - \Big(B(\Gamma_{y + w}, \rho) \cap B(\Gamma_w, \rho)\Big)\]
for all $y \in H$.\\ 

Bohr sets are symmetric, thus 
\[\Big(B(\Gamma_{y + z}, \rho) \cap B(\Gamma_z, \rho)\Big) - \Big(B(\Gamma_{y + w}, \rho) \cap B(\Gamma_w, \rho)\Big) = B(\Gamma_{y + z} \cup \Gamma_z, \rho) +  B(\Gamma_{y + w} \cup \Gamma_w, \rho).\]

If we were in a vector space rather than an arbitrary abelian group, the sum of Bohr sets above would become a sum of `orthogonal complements' $\langle \Gamma_{y + z} \cup \Gamma_z \rangle^\perp + \langle\Gamma_{y + w} \cup \Gamma_w \rangle^\perp$ which would be equal to $(\langle\Gamma_{y + z} \cup \Gamma_z\rangle \cap \langle \Gamma_{y + w} \cup \Gamma_w\rangle)^\perp$. We therefore use Theorem~\ref{bohrSum} which generalizes this fact to arbitrary abelian groups and provides us with a positive integer $R \leq \exp\Big(\log^{O(1)} \delta^{-1}\Big)$ such that\footnote{Again, $R$ can be taken to be independent of $y,z,w$ as increasing $R$ is allowed and the bound Theorem~\ref{bohrSum} is uniform.}
\[B(\langle \Gamma_{y + z} \cup \Gamma_z \rangle_R \cap \langle \Gamma_{y + w} \cup \Gamma_w \rangle_R; 1/4) \subseteq B(\Gamma_{y + z} \cup \Gamma_z; \rho) +  B(\Gamma_{y + w} \cup \Gamma_w, \rho).\]

Define $U_y = \langle \Gamma_{y} \rangle_R$ for each $y \in Y^1$. In conclusion, for each $y \in H$ we have
\[A^3_{\bcdot y} \supseteq \bigcup_{\ssk{z, w \in Y^1\text{ s.t.}\\y + z, y + w \in Y^1}} B\Big((U_{y+z} + U_z) \cap (U_{y+w} + U_w); 1/4\Big).\]

Applying Corollary~\ref{linearCoverCor} (with choices $G = H$, $H = \hat{G}$, $A = Y^1$, where the objects on the left sides of these three equalities are the ones in the statement of the corollary\footnote{There is a technical condition that $|G| \geq 2000 K^4 \alpha^{-4}$. If the condition fails, that means that $|H| \leq \exp\Big(O(\delta^{-O(1)})\Big)$ in which case Theorem~\ref{bogruzsabilinear2} reduces to the one-dimensional variant, i.e. Theorem~\ref{bogRuzsa1}.}), we obtain a positive integer $m \leq \exp(\log^{O(1)} \delta^{-1})$, proper coset progressions $C_1, \dots, C_m$ of rank at most $\log^{O(1)} \delta^{-1}$ and Freiman homomorphisms $L_1 \colon C_1 \to \hat{G}, \dots, L_m \colon C_m \to \hat{G}$ such that, writing $U'_a = \{L_i(a) \colon i \in [m], a \in C_i\} \cap U_a$ for each $a \in Y^1$, we have $y + z, z, y + w, w \in Y^1$ and 
\[(U_{y + z} - U_z) \cap (U_{y + w} - U_w) \subseteq (U'_{y + z} - U'_z) + (U'_{y + w} - U'_w)\]
for at least $\Omega(\delta^4|H|^3)$ of $y,z,w \in H$. Thus, we get a collection $T$ of triples $(y,z,w) \in H^3$ of size at least $\Omega(\delta^4|H|^3)$ such that $y + z, z, y + w, w \in Y^1$ and
\[A^3_{\bcdot y} \supseteq B\Big((U'_{y + z} - U'_z) + (U'_{y + w} - U'_w); 1/4\Big).\]

By Theorem~\ref{quantLattice} for each $y \in Y^1$, we may find elements $\gamma_1, \dots, \gamma_\ell \in U'_y$ such that $U'_y \subseteq\langle \gamma_1, \dots, \gamma_\ell\rangle_S$, for some $\ell \leq \log^{O(1)} \delta^{-1}$ and $S \leq \exp\Big(\log^{O(1)} \delta^{-1}\Big)$. In particular, there is a set of indices $I_y = \{i_1, \dots, i_\ell\} \subseteq [m]$ such that $\gamma_j = L_{i_j}(y)$ and $y \in C_{i_j}$. Let $\ell_0$ and $S_0$ be the maximal values attained by $\ell$ and $S$ as $y$ ranges over $Y^1$ which still satisfy $\ell_0 \leq \log^{O(1)} \delta^{-1}$ and $S_0 \leq \exp\Big(\log^{O(1)} \delta^{-1}\Big)$.\\
\indent We have now reached the point in the argument where we need an idea of Hosseini and Lovett~\cite{HosseiniLovett} in order to prevent the loss of efficiency in the argument. Observe that bounds we are after have two forms: one is $\log^{O(1)} \delta^{-1}$, which we want for quantities such as the rank of coset progressions and number of Freiman homomorphisms, and the other is $\exp\Big(\log^{O(1)} \delta^{-1}\Big)$ which is typically the inverse of the density of relevant sets in $G \times H$. Corollary~\ref{linearCoverCor} has the problem that it provides us with the exponential number of Freiman homomorphisms. Hosseini and Lovett use the fact that only $\log^{O(1)} \delta^{-1}$ of these maps are needed to cover any of the sets $U_y$, and instead average over these small subsets of Freiman homomorphisms. This turns the exponential bound into the inverse of the density of the relevant set rather than being a bound on the number of maps, as desired.\\
\indent Let us now return to the argument. By averaging, we may find four sets $J_1, J_2, J_3, J_4 \subseteq [m]$ such that for at least $\Omega\Big(\frac{1}{\binom{m}{\ell_0}^4}\delta^4|H|^3\Big)$ of triples $(y,z,w) \in T$ we have $I_{y+z} = J_1, I_z = J_2, I_{y + w} = J_3, I_w = J_4$. Furthermore, by averaging we may find $z,w \in Y^1$ and a set $Y^2$ of size at least $\Omega\Big(\frac{1}{\binom{m}{\ell_0}^4}\delta^4|H|\Big)$ whose elements $y$ satisfy $(y,z,w) \in T$ and $I_{y+z} = J_1, I_z = J_2, I_{y + w} = J_3, I_w = J_4$. We claim that we may find the following structure inside $A^3$.

\begin{claim}For each $y \in Y^2$ we have
\[A^3_{\bcdot y} \supseteq B\Big(\{L_i(y + z) \colon i \in J_1\} \cup \{L_i(z) \colon i \in J_2\} \cup \{L_i(y + w) \colon i \in J_3\} \cup \{L_i(w) \colon i \in J_4\}; 1/(16 \ell_0 S_0)\Big)\]
and all points in the arguments of maps $L_i$ defining the displayed Bohr set lie in the relevant coset progressions $C_i$.\label{a3y2claim}
\end{claim}

\begin{proof}Let $y \in Y^2$ and suppose that $x$ is an element of the Bohr set
\[B\Big(\{L_i(y + z) \colon i \in J_1\} \cup \{L_i(z) \colon i \in J_2\} \cup \{L_i(y + w) \colon i \in J_3\} \cup \{L_i(w) \colon i \in J_4\}; 1/(16 \ell_0 S_0)\Big).\]
By our assumptions, we have $I_{y+z} = J_1$ which implies $U'_{y + z} \subseteq \langle L_{i}(y + z) \colon i \in J_1\rangle_{S_0}$ and $y + z \in C_{i}$ for $i \in J_1$. Thus, for each $\chi \in U'_{y + z}$, we have $\chi = \sum_{i \in J_1} \lambda_i L_{i}(y + z)$ for integers $\lambda_i$ such that $|\lambda_i| \leq S_0$ for $i \in J_1$ so
\[|\chi(x)|_{\mathbb{T}} \leq \sum_{i \in J_1} |\lambda_i| \Big|L_{i}(y + z)(x)\Big|_{\mathbb{T}} \leq S_0  |J_1| \frac{1}{16 \ell_0 S_0} \leq \frac{1}{16}.\]
Thus, if $\theta \in (U'_{y + z} - U'_z) + (U'_{y + w} - U'_w)$ then we have $|\theta(x)|_{\mathbb{T}} \leq 1/4$, so 
\[x \in B\Big((U'_{y + z} - U'_z) + (U'_{y + w} - U'_w); 1/4\Big) \subseteq A^3_{\bcdot y}.\qedhere\]
\end{proof}

Let $\Gamma = \{L_i(z) \colon i \in J_2\} \cup \{L_i(w) \colon i \in J_4\} \subseteq \hat{G}$. For each $y \in Y^2$ we have that $y \in C_i - z$ for $i \in J_1$ and $y \in C_i - w$ for $i \in J_3$. Thus, $Y^2 \subseteq \Big(\cap_{i \in J_1} C_i - z\Big) \cap \Big(\cap_{i \in J_3} C_i - w\Big)$. Applying Proposition~\ref{cprogintersection} we obtain a proper coset progression $C \subseteq \Big(\cap_{i \in J_1} C_i - z\Big) \cap \Big(\cap_{i \in J_3} C_i - w\Big)$ of rank $d \leq \log^{O(1)} \delta^{-1}$ such that $|C \cap Y^2| \geq \exp\Big(-\log^{O(1)} \delta^{-1}\Big)|H|$. Furthermore, we may define Freiman homomorphisms $L'_i \colon C \to \hat{G}$ for $i \in J_1$ by $L'_i(y) = L_i(y + z)$ and $L''_i \colon C \to \hat{G}$ for $i \in J_3$ by $L''_i(y) = L_i(y + w)$. Write $Y^3 = Y^2 \cap C$.\\
\indent Misusing the notation by writing $L_1, \dots, L_r$ instead of $L'_i$ for $i \in J_1$ and $L''_i$ for $i \in J_3$, we get Freiman homomorphisms $L_i \colon C \to \hat{G}$ for $i \in [r]$, where $r \leq \log^{O(1)} \delta^{-1}$, a subset $\Gamma \subseteq \hat{G}$ of size $|\Gamma| \leq \log^{O(1)} \delta^{-1}$, a subset $Y^3 \subseteq C$ of size $|Y^3| \geq \exp\Big(-\log^{O(1)} \delta^{-1}\Big)|H|$ and a positive quantity $\rho \geq \exp\Big(-\log^{O(1)} \delta^{-1}\Big)$ such that for each $y \in Y^3$ we have
\[A^3_{\bcdot y} \supseteq B(\Gamma \cup \{L_1(y), \dots, L_r(y)\}; \rho)\]
using Claim~\ref{a3y2claim}.\\

\noindent\textbf{Linearization.} For technical reasons, we need to assume that the coset progression $C$ is symmetric and that maps $L_i$ are defined on $2C - 2C$ instead of just on $C$ and that are Freiman-linear rather than merely Freiman homomorphisms. We now make a slight digression from the main argument in order to achieve this.\\

Let $C = a + [0, N_1] \cdot v_1 + \dots + [0, N_d] \cdot v_d + H_0$ be a canonical form of $C$ and let $\delta_3 = |Y^3|/|H|$. Let $D_j = \Big\lfloor \frac{\delta_3 N_j}{100d}\Big\rfloor$ for $j \in [d]$, when $N_j \geq 200 d {\delta_3}^{-1}$ (so that $D_j \geq 2$), and $D_j = 0$ otherwise. Using the canonical form, sub-coset progression 
\[a + [4D_1, N_1-4D_1] \cdot v_1 + \dots + [4D_d, N_d-4D_d] \cdot v_d + H_0\]
of $C$ can be partitioned into further proper coset progressions\footnote{The old coset progressions denoted by $C_i$ no longer play a role in the proof, and we need the new coset progressions $C_1, \dots, C_M$ only temporarily, so we opted for this misuse of notation.} $C_1, \dots, C_M$ for $M \leq (200 {\delta_3}^{-1} d + 1)^d$ such that every $C_i$ has canonical form
\[C_i = a + [s_1, t_1] \cdot v_1 + \dots + [s_d, t_d] \cdot v_d + H_0\]
for suitable integers $s_1, t_1, \dots, s_d, t_d$ that satisfy
\begin{equation}\label{cosetprogpartreq}4D_j \leq s_j \leq t_j \leq N_j - 4D_j,\,\,D_j/2 \leq t_j - s_j \leq D_j\,\,\text{ and }\,\,t_j - s_j\,\,\text{ is even (possibly zero if $D_j = 0$).}\end{equation}
We then have $|C \setminus (\cup_{j \in [M]} C_j)| \leq \delta_3/2 |C|$ so by averaging, there is $j \in [M]$ such that $|C_j \cap Y^3| \geq \delta_3/(2M) |H|$. Since $t_\ell - s_\ell$ is even for all $\ell$, we may find an element $t \in C_j$ such that $\tilde{C} = C_j - t$ is a symmetric proper coset progression. Due to requirements~\eqref{cosetprogpartreq} we also know that $4\tilde{C} + t \subseteq C$. Define maps $\tilde{L}_1, \dots, \tilde{L}_r \colon C - t \to \hat{G}$ by $\tilde{L}_i(x) = L_i(x + t) - L_i(t)$, which are Freiman-linear. Finally, set
\[\tilde{A} = \Big(G \times ((Y^3 - t) \cap \tilde{C})\Big) \cap \Big(\bigcup_{y \in \tilde{C}} B(\Gamma \cup \{L_1(t), \dots, L_r(t)\} \cup \{\tilde{L}_1(y), \dots, \tilde{L}_r(y)\}; \rho/2)\Big).\]

We prove the following claim.

\begin{claim}We have $\tilde{A} + (0, t) \subset A^3$.\end{claim}

\begin{proof}Suppose that $(x, y) \in \tilde{A}$. Then $y + t \in Y^3$ and $x \in B(\Gamma; \rho)$. Furthermore, for $\ell \in [r]$, we have
\[\tn{L_i(y + t)(x)} \leq \tn{L_i(t)(x)} + \tn{L_i(y + t)(x) - L_i(t)(x)} = \tn{L_i(t)(x)} + \tn{\tilde{L}_i(y)(x)} \leq \rho,\]
showing that $x \in B(\Gamma \cup \{L_1(y + t), \dots, L_r(y + t)\}; \rho) \subseteq A^3_{\bcdot y + t}$.\end{proof}

In particular, $\dver \dver \tilde{A} \subseteq \dver \dver A^3$, so we may focus our attention on $\tilde{A}$ instead of $A^3$. We misuse the notation again, writing $A^3$ instead of $\tilde{A}$, $C$ instead of $\tilde{C}$, $Y^3$ instead of $(Y^3 - t) \cap \tilde{C}$, $L_i$ instead of $\tilde{L}_i$, $\Gamma$ instead of $\Gamma \cup \{L_1(t), \dots, L_r(t)\}$, etc. Thus, we may assume that $Y^3 \subseteq C$ for a symmetric proper coset progression $C$ of rank $d \leq \log^{O(1)} \delta^{-1}$ with Freiman-linear maps $L_1, \dots, L_r \colon 4C \to \hat{H}$ for $r \leq \log^{O(1)} \delta^{-1}$ and that
\begin{equation}A^3_{\bcdot y} \supseteq B(\Gamma \cup \{L_1(y), \dots, L_r(y)\}; \rho)\label{a3inclneqn}\end{equation}
for all $y \in Y^3$. We also have bounds $|Y^3| \geq \exp\Big(-\log^{O(1)} \delta^{-1}\Big)|H|$, $|\Gamma| \leq \log^{O(1)} \delta^{-1}$ and $\rho \geq \exp\Big(-\log^{O(1)} \delta^{-1}\Big)$.\\

\noindent\textbf{Step 4.} Let $\alpha$ be the density of $Y^3$ inside $C$ and recall that $d$ is the rank of $C$. We have $\alpha \geq \exp\big(-\log^{O(1)} \delta^{-1}\big)$. Let $\eta > 0$ be a parameter to be chosen later. Apply Theorem~\ref{algreglemma} to the bilinear Bohr variety (note the somewhat smaller radius)
\[W = G \times C \cap \Big(\bigcup_{y \in C} B(\Gamma \cup \{L_1(y), \dots, L_r(y)\}; \rho / 4) \times \{y\}\Big)\]
to obtain a partition of $C = C_1 \cup \dots \cup C_m$ into further proper coset progressions of rank at most $d$ such that the intersection with each $G \times C_i$ is quasirandom with parameter $\eta$ in the sense of Theorem~\ref{algreglemma}, where $m \leq \exp\big(\log^{O(1)} (\delta^{-1} \eta^{-1})\big)$. By averaging, we may find a piece $C' = C_i$ such that $|C'| \geq \frac{\alpha}{2m}|C|$ and $|C' \cap Y^3| \geq \frac{\alpha}{2} |C'|$. By Corollary~\ref{robustBogRuzsaCP} applied to the set $Y' = Y^3 \cap C'$ inside the progression $C'$, there is a proper symmetric coset progression $D \subseteq H$ of rank at most $\log^{O(1)}(\alpha^{-1} 2^{d + 1}) \leq \log^{O(1)} \delta^{-1}$ and size $|D| = \beta |C'|$ for some $\beta \geq \exp\big(-\log^{O(1)} \delta^{-1}\big)$ such that for each $y \in D$ there are at least $\frac{\alpha}{2^{d + 6}}|Y'|^3$ quadruples $(y_1, y_2, y_3, y_4)$ of elements in $Y'$ such that $y_1 + y_2 - y_3 - y_4 = y$. For $y \in H$, write $Q_y$ for the set of triples $(y_1, y_2, y_3)$ of elements in $Y'$ such that $y_1 + y_2 - y_3 - y \in Y'$. Thus, whenever $y \in D$ we have $|Q_y| \geq \exp\big(-\log^{O(1)} \delta^{-1}\big)|C'|^3$.\\
\indent By the conclusion of Theorem~\ref{algreglemma} we have that the bilinear Bohr variety $(G \times C') \cap W'$ where
\[W' = \Big(\bigcup_{y \in C} B(\Gamma \cup \{L_1(y), \dots, L_r(y)\}; \rho') \times \{y\}\Big)\]
is quasirandom in the sense of Theorem~\ref{algreglemma} with parameter $\eta$ and density $\delta'$, for some $\delta' \in [0,1]$ and radius $\rho' \in [\rho/8, \rho/4]$.\\
\indent Write $A^4 = \dver\dver A^3$.\\

\begin{claim}\label{wprimetoa4claim}Suppose that $(x,y) \in W'$. Assume also that $Q_y \cap \Big(W'_{x \bcdot}\Big)^3 \not= \emptyset$. Then $(x,y) \in A^4$.\end{claim}

\begin{proof}[Proof of the claim] Let $(y_1, y_2, y_3) \in Q_y \cap \Big(W_{x \bcdot}\Big)^3$. Write $y_4 = y_1 + y_2 - y_3 - y$ which we know to belong to $Y'$. Note that $y = y_1 + y_2 - y_3 - y_4 \in 4C$. We also know that $x \in B(\Gamma; \rho)$. Let $i \in [r]$. Then
\begin{align*}|L_i(y_4)(x)|_{\mathbb{T}} = &|L_i(y_1 + y_2 - y_3 - y)(x)|_{\mathbb{T}}\\
= &|L_i(y_1 + y_2 - y_3)(x) - L_i(y)(x)|_{\mathbb{T}}\\
= &|L_i(y_1 + y_2)(x) - L_i(y_3)(x) - L_i(y)(x)|_{\mathbb{T}}\\
= &|L_i(y_1)(x) + L_i(y_2)(x) - L_i(y_3)(x) - L_i(y)(x)|_{\mathbb{T}}\\
 \leq &|L_i(y_1)(x)|_{\mathbb{T}} + |L_i(y_2)(x)|_{\mathbb{T}} + |L_i(y_3)(x)|_{\mathbb{T}} + |L_i(y)(x)|_{\mathbb{T}} \leq \rho,\end{align*}
where we used the fact that the Freiman-linear map $L_i$ is defined on $4C$ so that all the points appearing as arguments of $L_i$ above belong to its domain. Thus $|L_i(y_4)(x)|_{\mathbb{T}} \leq \rho$, so $(x, y_4) \in A^3$ by~\eqref{a3inclneqn}. Since we already know that $(x,y_1), (x,y_2), (x,y_3) \in A^3$, we are done.\end{proof}

Recall that $\rho' \geq \rho / 8$ and thus $|B(\Gamma; \rho')| \geq (\rho/8)^{|\Gamma|}|G|$ by Lemma~\ref{basicbohrsizel}. Let us also recall what the conclusion of Theorem~\ref{algreglemma} implies about the structure of the bilinear Bohr variety $(G \times C') \cap W'$, namely the following.
\begin{itemize}
\item[\textbf{(i)}] For at least a $1 - \eta$ proportion of all elements $y \in C'$ we have 
\[\Big||B(\Gamma \cup \{L_1(y), \dots, L_r(y)\}; \rho')| - \delta' |B(\Gamma; \rho')|\Big| \leq \eta (8 \rho^{-1})^{|\Gamma|} |B(\Gamma; \rho')|.\]
\item[\textbf{(ii)}] For at least a $1 - \eta$ proportion of all pairs $(y, y') \in C' \times C'$ we have
\[\Big||B(\Gamma \cup \{L_1(y), \dots, L_r(y), L_1(y'), \dots, L_r(y')\}; \rho')| - {\delta'}^2 |B(\Gamma; \rho')|\Big| \leq \eta (8 \rho^{-1})^{|\Gamma|} |B(\Gamma; \rho')|.\]
\end{itemize}
Now we use quasirandomness theory as described in Appendix~\ref{qrAppendix}. First of all, using Lemma~\ref{appendonesided} we deduce that the bipartite graph on vertex classes $B(\Gamma; \rho')$ and $C'$ whose edges are given by pairs in $W'$ is $\eta'$-quasirandom for $\eta' = 3\sqrt[8]{\eta (8 \rho^{-1})^{|\Gamma|} + \eta}$ with density $\delta_{\text{qr}}$ that obeys $|\delta_{\text{qr}} - \delta'| \leq \eta(8 \rho^{-1})^{|\Gamma|} + \eta$. Provided $\eta \leq \frac{1}{100} (\rho /8)^{2|\Gamma| + r}$, using Lemma~\ref{basicbohrsizel} for the Bohr set $|B(\Gamma \cup \{L_1(y), \dots, L_r(y)\}; \rho')|$ this time, we see that $\delta_{\text{qr}} \geq \frac{1}{4}(\rho/8)^{|\Gamma| + r}$.\\

There is a positive quantity $\xi \geq \exp\Big(-\log^{O(1)} \delta^{-1}\Big)$ such that $|Q_y| \geq \xi |C'|^3$ holds for all $y \in D$. Take any $y \in D$.  Apply Lemma~\ref{genappendknhoods} to $Q_y$ to conclude that 
\[\Big||(W'_{x \bcdot})^3 \cap Q_y| - \delta_{\text{qr}}^3 |Q_y|\Big| \leq \frac{1}{2}\delta_{\text{qr}}^3 \xi |C'|^3 \leq \frac{1}{2}\delta_{\text{qr}}^3 |Q_y|\]
holds for all but at most $50 \delta_{\text{qr}}^{-6} \xi^{-2} \eta'$ proportion of elements $x \in B(\Gamma; \rho')$. Lemma~\ref{basicbohrsizel} tells us that $|W'_{\bcdot y}| \geq (\rho/8)^{|\Gamma| + r} |G| \geq (\rho/8)^{|\Gamma| + r} |B(\Gamma; \rho')|$. Therefore, $(W'_{x \bcdot})^3 \cap Q_y \not= \emptyset$ for all but at most $\eta'' |W'_{\bcdot y}|$ elements $x \in W'_{\bcdot y}$, where
\[\eta'' = 50 \delta_{\text{qr}}^{-6} \xi^{-2} (8/\rho)^{|\Gamma| + r} \eta'.\]
Combining this fact with Claim~\ref{wprimetoa4claim}, we conclude that for each $y \in D$ the row $A^4_{\bcdot y}$ has size
\[|A^4_{\bcdot y}| \,\geq (1 - \eta'')|B(\Gamma \cup \{L_1(y), \dots, L_r(y)\}; \rho')|.\]
\vspace{\baselineskip}
\textbf{Step 5.} Let $A^5 = \dhor A^4$. To finish the proof, we apply Lemma~\ref{almostfullBohr} for each $y \in D$. In order to make the set $A^4_{\bcdot y}$ sufficiently dense in the Bohr set above, we may pick $\eta \geq \exp\Big(-\log^{O(1)} \delta^{-1}\Big)$ (where the new implicit constants hidden in the $O(1)$ notation are sufficiently large in terms of the previous ones) to guarantee $\eta'' \leq 4^{-|\Gamma| - r - 1}$. We conclude that
\[\bigcup_{y \in D} \Big(B(\Gamma \cup \{L_1(y), \dots, L_r(y)\}; \rho'/2) \times \{y\}\Big) \subseteq A^5,\]
completing the proof of the theorem.\end{proof}

\appendix 
\section*{Appendix}

\section{Robust Bogolyubov-Ruzsa lemma}\label{robbogruzapp}

In this appendix we provide a quick proof of the robust version of the Bogolyubov-Ruzsa lemma, which was stated as Theorem~\ref{robustBogRuzsa}. We recall its statement here.

\begin{theorem}[Theorem~\ref{robustBogRuzsa}]\label{robustAppendix}Let $G$ be a finite abelian group. Let $A \subset G$ be a set such that $|A + A| \leq K |A|$. Then there exists a symmetric proper coset progression $C$ of rank at most $O(\log^{O(1)} (2K))$ and size $|C| \geq \exp(-O(\log^{O(1)} (2K)))|A|$ such that for each $x \in C$ there are at least $\frac{1}{64K}|A|^3$ quadruples $(a_1, a_2, a_3, a_4) \in A^4$ such that $x = a_1 + a_2 - a_3 - a_4$.\end{theorem}

We deduce this theorem from the following result of Schoen and Sisask (see Theorem 5.1 in~\cite{SchSisRob} which is applied with the choice $M = S = A$ in their notation to yield the formulation below).

\begin{theorem}\label{schoenSisaskResult}Let $G$ be a finite abelian group and let $\varepsilon > 0$. Suppose that $A, L \subseteq G$ are two sets and that $|A + A| \leq K |A|$. Assume also that $\eta = |A|/|L| \leq 1$. Then there exists a set $T \subseteq A$ of size
\[|T| \geq \exp(-O(\varepsilon^{-2} \log (2K) \log (2\eta^{-1}))) |A|\]
such that
\[\Big|\id_A \ast \id_A \ast \id_L(x + t) - \id_A \ast \id_A \ast \id_L(x)\Big| \leq \varepsilon |G|^{-2} |A|^2\]
holds for all $x \in G$ and $t \in 2T - 2T$.\end{theorem}

\begin{proof}[Proof of Theorem~\ref{robustAppendix}] Let $S$ be the set of all elements $s \in A + A$ such that $s = a + b$ for at least $\frac{1}{2K} |A|$ pairs $(a,b) \in A^2$. Since $|A + A| \leq K|A|$, the number of pairs $(a,b) \in A^2$ such that $a + b \notin S$ is at most $\frac{1}{2}|A|^2$. Thus, $|\{(a,b) \in A^2 \colon a + b \in S\}| \geq \frac{1}{2}|A|^2$, so by averaging, there is an element $a \in A$ such that $|a + A \cap S| \geq \frac{1}{2}|A|$, implying in particular that $|S| \geq |A| / 2$.\\ 
\indent Let $A' \subseteq A$ be a subset of size $\lfloor|A|/2\rfloor$ chosen randomly and uniformly among all subsets of that size. Then 
\begin{align*}\exx \id_{A'} * \id_{A'} * \id_{-S}(0) =& |G|^{-2}\exx |\{(a,b) \in A' \times A' \colon a + b \in S\}|\\
 = &|G|^{-2} \sum_{\ssk{(a,b) \in A \times A\\a + b \in S}} \mathbb{P}(a, b \in A')\\
\geq & \frac{1}{16} |G|^{-2} |A|^2,\end{align*}
provided $|A|$ is larger than some sufficiently large absolute constant.\footnote{We need this condition on the size of $A$ so that that the probability of a pair of elements $a,b \in A$ ending up in $A'$ is at least $1/8$. On the other hand, the theorem holds trivially for $A$ of bounded size as we can take $C = \{0\}$ since there are always at least $\frac{|A|^3}{K}$ solutions to $0 = a_1 + a_2 - a_3 - a_4$ with elements $a_1, \dots, a_4 \in A$.}\\
\indent Thus, there is a choice of $A'$ satisfying this bound. Fix such a subset $A' \subseteq A$. Let $\eta = |A'|/|S|$. Note also that $\frac{1}{2K} \leq \eta \leq 1$.\\
\indent Using Theorem~\ref{schoenSisaskResult} with sets $A'$ and $-S$ and parameter $\varepsilon = \frac{1}{32}$, we see that
\[\id_{A'} \ast \id_{A'} \ast \id_{-S}(t) \geq \frac{1}{32}|G|^{-2}|A|^2\]
holds for all $t \in 2T - 2T$, where $T \subseteq A'$ is a set of size at least $\exp\Big(-O(\log^{O(1)} (2K)))\Big)|A|$. Note that $|T + T| \leq |A' + A'| \leq K |A| \leq \exp(O(\log^{O(1)} (2K)))|T|$. Apply Theorem~\ref{bogRuzsa1} to find a symmetric proper coset progression $C \subset 2T - 2T$ of rank at most $O(\log^{O(1)} (2K))$ and size at least $\exp(-O(\log^{O(1)} (2K))) |A|$. For each $c \in C$ we have at least $\frac{1}{64K}|A|^3$ of quadruples $(a_1, a_2, a_3, a_4) \in A^4$ such that $c = a_1 + a_2 - a_3 - a_4$.\end{proof} 

\section{Quasirandomness of bipartite graphs}\label{qrAppendix}

Throughout this appendix, $G$ denotes a bipartite graph on vertex classes $X$ and $Y$. We also view $G$ simultaneously as a $\{0,1\}$-valued function on the product $X \times Y$. Let $\|\cdot\|_{\square} = \|\cdot\|_{\square(X, Y)}$ stand for the box norm, which is defined by 
\[\|f\|_{\square(X, Y)} = \Big(\exx_{x_0, x_1 \in X} \exx_{y_0, y_1 \in Y} f(x_0, y_0)\,\overline{f(x_1, y_0)}\,\overline{f(x_0, y_1)}\,f(x_1, y_1)\Big)^{1/4}.\]

An important property of the box norm is that it provides bounds on the correlation of a given function $f$ with functions depending on a single variable only.

\begin{lemma}\label{basicgcs}Let $f \colon X \times Y \to \mathbb{C}$, $u \colon X \to \mathbb{C}$ and $v \colon Y \to \mathbb{C}$ be functions. Then
\[\Big|\exx_{x \in X, y \in Y} f(x,y)u(x)v(y) \Big| \leq \|f\|_\square \|u\|_{L^2} \|v\|_{L^2}.\]
\end{lemma}

Given a bipartite graph $G$ on vertex classes $X$ and $Y$, its density $\delta$ is given by $\ex_{x \in X, y \in Y} G(x,y)$. We say that $G$ is $\varepsilon$-\emph{quasirandom} if $\|G - \delta\|_\square \leq \varepsilon$. Quasirandom graphs behave like randomly chosen graphs of the given density. An instance of this phenomenon is given by the following lemma, which says that the intersection of neighbourhoods of $k$ vertices in one vertex class typically has density about $\delta^k$ in the other vertex class.

\begin{lemma}\label{appendknhoods}Let $k$ be a positive integer and let $G$ be an $\varepsilon$-quasirandom bipartite graph of density $\delta$ on vertex classes $X$ and $Y$. Pick $x_1, \dots, x_k \in X$ uniformly and independently at random. Let $\eta > 0$ be a positive real. Then 
\[\mathbb{P}\Big(\Big||N_{x_1} \cap \dots \cap N_{x_k}| - \delta^k |Y|\Big| \geq \eta |Y|\Big) \leq 4k \eta^{-2} \varepsilon.\]
\end{lemma}

In our paper we need a more general version of this lemma, which we now state and prove. Lemma~\ref{appendknhoods} follows as a corollary.

\begin{lemma}\label{genappendknhoods}Let $k$ and $m$ be positive integers and let $G$ be an $\varepsilon$-quasirandom bipartite graph of density $\delta$ on vertex classes $X$ and $Y$. Let $M \subseteq Y^m$ be a set of $m$-tuples in $Y$. Pick $x_1, \dots, x_k \in X$ uniformly and independently at random. Let $\eta > 0$ be a positive real. Then 
\[\mathbb{P}\Big(\Big||N^m_{x_1} \cap \dots \cap N^m_{x_k} \cap M| - \delta^{mk} |M|\Big| \geq \eta |Y|^m\Big) \leq  4km \eta^{-2}\varepsilon.\]\end{lemma}

\begin{proof}Let $f(x,y) = G(x,y) - \delta$ and $\mu = |M|/|Y|^m$. We have
\begin{align}\frac{1}{|Y|^{2m}} \exx_{x_1, \dots, x_k \in X} &\Big||N^m_{x_1} \cap \dots \cap N^m_{x_k} \cap M| - \delta^{mk} |M|\Big|^2\nonumber\\
=  &\exx_{x_1, \dots, x_k \in X} \Big|\exx_{y_1, \dots, y_m \in Y} M(y_1, \dots, y_k) \Big(\prod_{i \in [k], j \in [m]}G(x_i, y_j) - \delta^{km}\Big)\Big|^2\nonumber\\
= & \exx_{\ssk{x_1, \dots, x_k \in X\\y_1, \dots, y_m \in Y\\z_1, \dots, z_m \in Y}} M(y_1, \dots, y_k) M(z_1, \dots, z_k) \prod_{i \in [k], j \in [m]}G(x_i, y_j)G(x_i, z_j)\nonumber\\
&\hspace{2cm} - 2\delta^{km}\mu \exx_{\ssk{x_1, \dots, x_k \in X\\y_1, \dots, y_m \in Y}} M(y_1, \dots, y_k) \prod_{i \in [k], j \in [m]}G(x_i, y_j) + \delta^{2km} \mu^2.\label{squareboundnhood}\end{align}

We approximate the term
\[\exx_{\ssk{x_1, \dots, x_k \in X\\y_1, \dots, y_m \in Y\\z_1, \dots, z_m \in Y}} M(y_1, \dots, y_k) M(z_1, \dots, z_k) \prod_{i \in [k], j \in [m]}G(x_i, y_j)G(x_i, z_j)\]
by $\delta^{2km}\mu^2$ using Lemma~\ref{basicgcs}. Let us write $M(y)$ and $M(z)$ for $M(y_1, \dots, y_k)$ and $M(z_1, \dots, z_k)$ respectively. Using triangle inequality, we get
\begin{align*}&\bigg|\exx_{\ssk{x_1, \dots, x_k \in X\\y_1, \dots, y_m \in Y\\z_1, \dots, z_m \in Y}} M(y)M(z)\Big(G(x_1, y_1) - \delta\Big) G(x_2, y_1) \cdots G(x_k, y_1) G(x_1, y_2) \cdots G(x_k, y_m)\Big(\prod_{i \in [k], j \in [m]}G(x_i, z_j)\Big)\bigg|\\
&\hspace{2cm}= \bigg|\exx_{\ssk{x_1, \dots, x_k \in X\\y_1, \dots, y_m \in Y\\z_1, \dots, z_m \in Y}} M(y)M(z) f(x_1, y_1) G(x_2, y_1) \cdots G(x_k, y_1) G(x_1, y_2) \cdots G(x_k, y_m)\Big(\prod_{i \in [k], j \in [m]}G(x_i, z_j)\Big)\bigg|\\
&\hspace{2cm}\leq \exx_{\ssk{x_2, \dots, x_k \in X\\y_2, \dots, y_m \in Y\\z_1, \dots, z_m \in Y}} \bigg| \exx_{\ssk{x_1 \in X\\y_1 \in Y}} f(x_1, y_1) \Big(M(y) G(x_2, y_1) \cdots G(x_k, y_1)\Big)\\
&\hspace{8cm} \Big(G(x_1, y_2)G(x_1, y_3)  \cdots G(x_1, y_m)G(x_1, z_1) \cdots G(x_1, z_m)\Big)\bigg|.\end{align*}

Applying Lemma~\ref{basicgcs} for every choice of $x_2, \dots, x_k, y_2, \dots, y_m, z_1, \dots, z_m$ above, we may bound the last line by $\|f\|_{\square}$. Thus
\begin{align*}&\bigg|\exx_{\ssk{x_1, \dots, x_k \in X\\y_1, \dots, y_m \in Y\\z_1, \dots, z_m \in Y}}M(y)M(z) \prod_{i \in [k], j \in [m]}G(x_i, y_j)G(x_i, z_j) \\
&\hspace{2cm}- \delta \exx_{\ssk{x_1, \dots, x_k \in X\\y_1, \dots, y_m \in Y\\z_1, \dots, z_m \in Y}} M(y)M(z) \prod_{\ssk{i \in [k], j \in [m]\\(i,j) \not= (1,1)}}G(x_i, y_j)\prod_{i \in [k], j \in [m]}G(x_i, z_j)\bigg| \leq \|f\|_\square.\end{align*}
Repeating the argument above $2km - 1$ times in order to eventually replace every occurrence of $G$ by $\delta$, we obtain
\[\bigg|\exx_{\ssk{x_1, \dots, x_k \in X\\y_1, \dots, y_m \in Y\\z_1, \dots, z_m \in Y}} M(y)M(z)\prod_{i \in [k], j \in [m]}G(x_i, y_j)G(x_i, z_j) - \delta^{2k}\mu^2\bigg| \leq 2km\|f\|_\square.\]

The term $\ex_{\ssk{x_1, \dots, x_k \in X\\y_1, \dots, y_m \in Y}} M(y) \prod_{i \in [k], j \in [m]}G(x_i, y_j)$ in expression~\eqref{squareboundnhood} can be similarly approximated by $\delta^{km}\mu$ 
\[\bigg|\exx_{\ssk{x_1, \dots, x_k \in X\\y_1, \dots, y_m \in Y}}\prod_{i \in [k], j \in [m]}G(x_i, y_j) - \delta^{km}\mu\bigg| \leq km \|f\|_{\square}.\]

Going back to~\eqref{squareboundnhood} and using the assumption $\|f\|_\square \leq \varepsilon$ we obtain
\begin{align*}\frac{1}{|Y|^{2m}} \exx_{x_1, \dots, x_k \in X} &\Big||N^m_{x_1} \cap \dots \cap N^m_{x_k} \cap M| - \delta^{mk} |M|\Big|^2 \leq 4km \varepsilon,\end{align*}
completing the proof.\end{proof}

As it is well-known, properties like the one described in Lemma~\ref{appendknhoods} are equivalent to the graph being quasirandom. An example of such a reverse implication is given by the next lemma.

\begin{lemma}\label{appendonesided}Let $\delta,\varepsilon \in [0,1]$. Suppose that $G$ is a bipartite graph with vertex classes $X$ and $Y$ such that
\begin{equation}\label{qrCond1}\exx_{x \in X} \Big||N_x| - \delta |Y|\Big| \leq \varepsilon |Y|\end{equation}
and 
\begin{equation}\exx_{x, x' \in X} \Big||N_x \cap N_{x'}| - \delta^2 |Y|\Big| \leq \varepsilon |Y|.\label{qrCond2}\end{equation}
Then the density $\delta'$ of $G$ satisfies $|\delta - \delta'| \leq \varepsilon$ and the graph $G$ is $3 \sqrt[8]{\varepsilon}$-quasirandom.\end{lemma}

\begin{proof}Let $f(x,y) = G(x,y) - \delta$. We begin the proof by showing that $\delta$ and $\delta'$ differ by a small amount. This follows easily from~\eqref{qrCond1}, namely
\[|\delta' - \delta| |X||Y| = \Big|\sum_{x \in X} \Big(|N_x| - \delta |Y|\Big)  \Big| \leq \sum_{x \in X} \Big||N_x| - \delta |Y|\Big| \leq \varepsilon |X||Y|.\]

By averaging, we conclude from condition~\eqref{qrCond1} that for at least $1 - \sqrt{\varepsilon}$ proportion of all $x \in X$
\[\Big||N_x| - \delta |Y|\Big| \leq  \sqrt{\varepsilon} |Y|\]
holds. Similarly, condition~\eqref{qrCond2} implies that for at least $1 - \sqrt{\varepsilon}$ proportion of all pairs $(x,x') \in X^2$ 
\[\Big||N_x \cap N_{x'}| - \delta^2 |Y|\Big| \leq \sqrt{\varepsilon} |Y|\]
holds. Thus, for at least $(1 - 3\sqrt{\varepsilon})|X|^2$ of pairs $(x,x') \in X^2$ we have 
\[\Big||N_x \cap N_{x'}| - \delta |N_x| - \delta|N_{x'}| + \delta^2 |Y|\Big| \leq \Big||N_x \cap N_{x'}| - \delta^2 |Y|\Big| + \Big|\delta|N_x| - \delta^2|Y|\Big| + \Big|\delta|N_{x'}| - \delta^2|Y|\Big| \leq 3\sqrt{\varepsilon} |Y|.\]
The term on the left hand side equals $|Y|$ times $\Big|\ex_{y \in Y} f(x,y) f(x', y)\Big|$, the latter being bounded from above by 1. Thus,
\[\|f\|_\square^4 = \exx_{x,x' \in X, y, y' \in Y} f(x,y) f(x', y) f(x,y') f(x', y') = \exx_{x, x' \in X} \Big|\exx_{y \in Y} f(x,y) f(x', y)\Big|^2 \leq 9\varepsilon + 3\sqrt{\varepsilon} \leq 12\sqrt{\varepsilon}.\]
Recalling that $|\delta - \delta'| \leq \varepsilon$, we obtain $\|G - \delta'\|_{\square} \leq 3 \sqrt[8]{\varepsilon}$.\end{proof}

We may view the condition in Lemma~\ref{appendonesided} as a type of one-sided quasirandomess. Thus, combining Lemmas~\ref{appendknhoods} and~\ref{appendonesided} we conclude that if $G$ satisfies the one-sided quasirandomness condition for the vertex class $X$, then it is in fact quasirandom, and thus is also satisfies the one-sided quasirandomness condition for the vertex class $Y$, with a somewhat smaller quasirandomness parameter.

\section*{Acknowledgments} 
I thank the anonymous referee for valuable feedback and useful comments.

\bibliographystyle{amsplain}


\begin{dajauthors}
\begin{authorinfo}[luka]
  Luka Mili\'cevi\'c\\
  Mathematical Institute of the Serbian Academy of Sciences and Arts\\
  Belgrade, Serbia\\
  luka\imagedot{}milicevic\imageat{}turing\imagedot{}mi\imagedot{}sanu\imagedot{}ac\imagedot{}rs\\
  \url{https://www.mi.sanu.ac.rs/~luka.milicevic/}
\end{authorinfo}
\end{dajauthors}

\end{document}